\documentclass[11pt,reqno]{amsart}
 \usepackage{amsmath}
 \usepackage{amsfonts}
\usepackage{amsthm}
\theoremstyle{plain}

\numberwithin{equation}{section}

\newtheorem{thm}{Theorem}[section]
\newtheorem{remark}[thm]{Remark}
\newtheorem{prop}[thm]{Proposition}

\newtheorem{cor}[thm]{Corollary}

\newtheorem{example}[thm]{Example}

\newcommand{\sbm}[1]{\left[\begin{smallmatrix} #1
		\end{smallmatrix}\right]}

\newcommand{\cD}{{\mathcal D}}

\newcommand{\cG}{{\mathcal G}}

\newcommand{\cL}{{\mathcal L}}
\newcommand{\cM}{{\mathcal M}}

\newcommand{\cO}{{\mathcal O}}

\newcommand{\cS}{{\mathcal S}}

\newcommand{\ba}{{\mathbf a}}
\newcommand{\be}{{\mathbf e}}
\newcommand{\bu}{{\mathbf u}}
\newcommand{\bk}{{\mathbf k}}
\newcommand{\bmu}{{\boldsymbol \mu}}
\newcommand{\blam}{{\boldsymbol \lambda}}
\newcommand{\bD}{{\mathbf D}}
\newcommand{\bM}{{\mathbf M}}

\begin{document}
 \title[Meromorphic trivializations]{Meromorphic matrix
trivializations of factors of 
 automorphy over a Riemann surface}
 \author[J. A.\ Ball]{Joseph A. Ball}
\address{Department of Mathematics,
Virginia Tech,
Blacksburg, VA 24061, USA}
\email{joball@math.vt.edu}
\author[K.F.\ Clancey]{Kevin F.\ Clancey}
\address{Department of Mathematics, University of Georgia,
Athens, GA 30602, USA}
\email{kclancey@uga.edu}
\author[V.\ Vinnikov]{Victor Vinnikov}
\address{Department of Mathematics, Ben Gurion University of the 
Negev, POB 653, 84105 Beer-Sheva, Israel}
\email{vinnikov@cs.bgu.ac.il}
\dedicatory{In memory of Leiba Rodman, a dear friend and dedicated 
colleague}

  \begin{abstract}  It is a consequence of the Jacobi Inversion
Theorem that a line bundle over a Riemann surface of genus $g$ has a meromorphic 
section having 
at most $g$ poles, or equivalently, the divisor class of a divisor
$D$ over $M$ contains a divisor  having at most 
$g$ poles (counting multiplicities).
We explore various analogues of these ideas for vector bundles and
associated matrix divisors over $M$.  The most 
explicit results are for the genus 1 case.  We also review and 
improve earlier results concerning the construction of automorphic or 
relatively automorphic  meromorphic matrix functions having a prescribed 
null/pole structure.
  \end{abstract}

  \maketitle
  
 \section{Introduction}
 Let $M$ be a closed Riemann surface of genus $g\geq 1$ and $\rho
:\widehat{M}\rightarrow M$ the universal cover with 
$\mathcal{G}$ the group of covering transformations.
 The collection
of equivalence classes of rank $r$ holomorphic vector 
bundles over $M$ is equivalent to the collection of equivalence
classes of rank $r$ holomorphic factors of automorphy on 
$\widehat{M}.$ Recall that a rank $r$ holomorphic factor of
automorphy on $\widehat{M}$ is a map 
$\zeta : \mathcal{G}\times\widehat{M}\rightarrow GL(r,\mathbb{C}),$
which is holomorphic on $\widehat{M}$  for $T$ 
fixed in $\mathcal{G},$  satisfying \[ \zeta (ST,u) = \zeta
(S,Tu)\zeta (T, u),\ \ \ S,T\in \mathcal{G},\ u\in\widehat{M}. \] 
Two rank $r$ factors of automorphy $\zeta$ and $\eta$ are
holomorphically equivalent in case there is a holomorphic function 
$h:\widehat{M} \rightarrow GL(r,\mathbb{C})$ satisfying 
\[h(Tu)\zeta (T,u) = \eta (T,u)h(u),\ \ \ T\in \mathcal{G},\ \
u\in\widehat{M}.\]
A rank $r$ factor of automorphy that is given by a representation
$\zeta :\mathcal{G}\rightarrow GL(r,\mathbb{C})$ 
is called a rank $r$ flat factor of automorphy.

Every rank $r$ factor of automorphy $\zeta$ on $\widehat{M}$ is
meromorphically trivial in the sense that there is an
$r\times r$-meromorphic
matrix function $\widehat F$ on $\widehat{M}$, 
nondegenerate in the sense that $\det \widehat F \not\equiv 0$,  such 
that 
\begin{equation} \label{trivial} 
    \widehat F(Tu)=\zeta (T,u) \widehat F(u),\, T\in\mathcal{G}, \,
u\in\widehat{M}. 
\end{equation} 
A nondegenerate $r\times r$-meromorphic matrix function $\widehat F$
satisfying (\ref{trivial}) will be called a 
    meromorphic matrix trivialization, or simply a trivialization, of
$\zeta.$ When \eqref{trivial} holds, we also say that $\widehat
F$ is a {\em relatively automorphic meromorphic matrix function} on
$\widehat M$ (with respect 
to the factor of automorphy $\zeta$).  In case the factor of 
automorphic is trivial (i.e., $\zeta(T,u) = I_{r}$ for all $T$ and 
$u$), we say simply that $F$ is an {\em automorphic meromorphic matrix 
function}.
If $F$ is a meromorphic matrix 
    trivialization of $\zeta,$ then any other meromorphic matrix
trivialization of $\zeta$ will have the form $FK,$ 
    where $K$ is a nondegenerate $r\times r$-automorphic meromorphic
matrix function on $\widehat{M}$, or what 
    amounts to the same,  a (single-valued) global meromorphic
function on $M$. Equivalently, the columns of a meromorphic matrix
trivialization 
form a basis for the space of meromorphic sections of $\zeta$ over
the field of meromorphic functions on $M.$

In the scalar case ($r=1$), there is a well-known correspondence
between 
three types of objects:  (1) factors of automorphy (flat or general) 
on $\widehat M$, (2) holomorphic  line bundles on $M$, and (3) 
divisors (i.e., pole/zero multiplicity specifications) on $M$.
While to some extent the analogue of this structure is understood for 
the higher rank case (see in particular \cite{Guntheta, Gunvect}),
one goal of this paper is to delineate more concretely
the higher rank analogue of this correspondence between (1) (left) 
matrix factors of automorphy on $\widehat M$, (2) holomorphic (row)
vector bundles on 
$M$, and (3) (right) matrix null-pole divisors (or right 
nondegenerate matrix germs) on $M$.   In particular, we continue the 
work of \cite{BC} by using the parametrization of a null-pole 
subspace  from \cite{BGR} to get a more explicit analogue of divisor 
(i.e., an encoding of zero and pole data) for the higher rank case.

If one starts with a divisor $\cD$ (i.e., an encoding of pole/zero 
structure which amounts to simply pole/zero multiplicity in the 
scalar case---precise definitions can be found in Section 
\ref{S:divisors} below), one can pose three basic Interpolation 
Problems:
\begin{enumerate}
    \item[(I) ]  \textbf{First Interpolation Problem:}
    Find an automorphic meromorphic matrix function $F$ on $M$ (i.e., 
    an automorphic meromorphic matrix function $\widehat F$ on the covering 
    surface $\widehat M$ with trivial factor of automorphy)  
    with right null/pole structure prescribed by $\cD$;
    
    \item[(II)]  \textbf{Second Interpolation Problem:}  Find a
    relatively automorphic  meromorphic matrix function $\widehat F$  on
$\widehat M$ with flat (left) factor of automorphy and with right null/pole 
    structure on $M$ prescribed by $\cD$;
    
    \item[(III)]  \textbf{Third Interpolation Problem:}
    Find a relatively automorphic meromorphic matrix function $\widehat F$ on 
    $\widehat M$ (with not necessarily flat factor of automorphy)
   with right null/pole structure on $M$ prescribed by $\cD$.
\end{enumerate}

Problem (I) was addressed in \cite{BC} for the case of ``simple'' 
null-pole structure (for precise definitions see Section 
\ref{S:simple} below) while Problem (II) was addressed in \cite{BCV}.
Here we revisit 
Problems (I) and (II) and complete the results obtained there to a 
more satisfactory form (see Theorem \ref{T:merint} below).  
We also make explicit the connections with the 
solution of a variant of Problem (I) studied in \cite{BV}.
It turns out that Problem (III) always has a 
solution; we give a simple direct proof of this result based on the 
theorem of Grauert \cite{Grauert} guaranteeing the triviality of any 
holomorphic vector bundle over a simply-connected domain (see the end 
of Section \ref{S:4.1} below).  

In the line bundle case ($r=1$), a well known consequence of 
Abel-Jacobi theory is that one can find a meromorphic trivialization 
$F$ of a flat line bundle having a divisor with total zero 
multiplicity equal to the pole multiplicity with multiplicity value 
equal to at most $g$ (the genus of $M$).  One purpose of this
paper is to find an analogue of this result for the vector bundle case, 
i.e., to analyze to what 
extent it is possible to find a meromorphic trivialization $F$  
of a given flat matrix factor of automorphy so that the matrix 
null-pole structure is as simple as possible (see Theorem 
\ref{T:NSF} and Corollary \ref{T:NSF'} below).  

In Section \ref{S:theta}  we present two explicit trivializations
of a flat factor of automorphy over an elliptic curve $M$ 
(i.e., Riemann surface of genus 1)
when the curve is presented in
the form $M=\mathbb{C}/(\mathbb{Z} + \tau\mathbb{Z})$. 
In Section \ref{S:simple}, we demonstrate how to construct
trivializations of a
flat factor of automorphy on an elliptic curve that 
has simple null-pole structure while maintaining a count on the
number of poles and zeros in such trivializations. These results 
use the fact that the uniformization of $M$ in the 
elliptic case is the complex plane ${\mathbb C}$ where one can make 
use of theta functions (including theta functions with matrix 
arguments) to give explicit constructions of relatively automorphic 
functions on $\widehat M = {\mathbb C}$. 

For generalizations to higher genus (the setting of Section \ref{S:divisors}), 
there are two distinct approaches which coalesce to a single approach 
for the case of genus $g$ equal to 1.
\begin{itemize}
    \item
The first approach uses the Abel-Jacobi map to 
embed the Riemann surface $M$ into the Jacobian variety ${\mathbb 
C}^{g}/\Lambda$ (where $\Lambda$ is the period lattice),  The 
universal cover is thereby embedded into ${\mathbb C}^{g}$ where one 
can use the function theory of ${\mathbb C}^{g}$ (in particular theta 
functions in $g$ variables). We use this approach in Subsection 
\ref{S:explicit} to obtain explicit formulas for the canonical 
functions used earlier in Section \ref{S:divisors} to adapt the theory of null/pole subspaces from 
\cite{BGR} to the Riemann-surface setting. From these explicit 
formulas one can read off continuous-dependence results which are 
needed to guarantee that a certain bundle-trivialization procedure 
leads to simple pole structure (Theorem \ref{T:NSF'} below).

\item  The second approach works directly with the uniformization 
$\widehat M$ which can be taken to be either the unit disk ${\mathbb 
D}$ or the upper or right half plane ${\mathbb H}$ contained in 
${\mathbb C}$, with the group of deck transformations concretely 
identified with  a Fuchsian group of linear-fractional transformations on $\widehat M$. 
The analogue of a theta function (constructed by adding up over all 
terms obtained via the action of a group of M\"obius transformations)
is what is called a Poincar\'e series (see e.g.\ \cite{B79, EM1, EM2, 
Forelli}).  It would be of interest to find higher-genus analogues 
of the genus-1 theta-function constructions in Sections \ref{S:theta} and \ref{S:simple} 
by using this second approach.
\end{itemize}
In the first approach an explicit description of the universal cover (an embedded 
submanifold in ${\mathbb C}^{g}$) is somewhat cumbersome while the 
group of deck transformations is simply the group of translations by 
elements of the period lattice; in the second approach the 
universal cover is simple but the group of deck transformations 
(linear-fractional maps rather than translations) is less explicit.  

The present study can be viewed as the latest installment in 
developing the higher-genus analogue  of the theory of interpolation 
for rational matrix functions (meromorphic matrix functions on the 
Riemann sphere ${\mathbb C} \cup {\infty}$, the essentially unique 
compact Riemann surface of genus $g=0$), a topic developed in much of 
the work of Leiba Rodman, especially in the monograph \cite{BGR}.  
It is with sadness and respect that we dedicate this 
paper to our dear friend and collaborator Leiba Rodman.

\section{Explicit trivializations of flat bundles when  $g=1$} 
    \label{S:theta}

\par In this section, we describe two natural meromorphic matrix
trivializations of a rank-$r$ flat factors of automorphy over 
an elliptic curve. In spite of the natural forms of these
trivializations, the associated meromorphic matrix functions do not
have simple null-pole divisor structure. This complication is
addressed in the next section.
\smallskip
\par \textbf{I. Theta\  function\  trivializations:}  Let $M$ be an
elliptic curve presented in 
the form $M=\mathbb{C} /( \mathbb{Z}+\tau \mathbb{Z})$, where as is
customary ${\rm Im} \,\tau >0.$   
The universal cover for $M$ can be taken to be $\widehat M = {\mathbb
C}$ and the group of deck 
transformations can be identified
with $G = {\mathbb Z} + \tau {\mathbb Z}$ in a natural way.  Then a
rank-$r$ flat vector bundle $\zeta$ over $M$ 
as discussed in the Introduction 
corresponds to a representation $\zeta :\mathbb{Z}+\tau 
\mathbb{ Z} \rightarrow  GL(r,\mathbb{ C})$. A flat factor of
automorphy $\zeta$ is equivalent to a normalized flat 
factor of automorphy $\xi_{V}$ where the representation has the form
\begin{equation}\xi_{V}(1)=I_{r}\ :\ \xi_{V}(\tau) = V, 
\label{rep}\end{equation} with $V$ an invertible $r\times r$-matrix.
The form of $V$ will be further simplified below.
Meromorphic matrix trivializations of $\xi_{V}$ correspond to
$r\times r$-meromorphic matrix functions $F$ on $\mathbb C$ 
that satisfy \[  F( u+m+n\tau ) = \xi_{V} (m+n\tau ) F(u)=V^{n}F(u).\]

\par\textbf{The\ scalar\ case:}  As is well known, in the case $r=1,$
meromorphic trivializations of  flat factors of 
automorphy on $\mathbb{C}$  (or, equivalently, line bundles) can be
given explicitly using theta functions. To see this, 
for $\alpha \neq 0,$ let $\xi_{\alpha}$ be the degree-zero line
bundle corresponding to the character 
$\xi_{\alpha}(m+\tau n) = \alpha^{n}$ and let $\theta$ be the
classical theta function defined for $u\in\mathbb{C}$ 
by 
\begin{equation}   \label{theta}
 \theta (u) = \sum_{n\in \mathbb{Z}} \exp2\pi
i\left(\frac{n^{2}\tau}{2} +n u\right). 
\end{equation}
The function $\theta$ 
has the following automorphic behavior
 \[\theta (u+1) = \theta (u)\ :\ \theta (u+\tau ) = \exp{ (-\pi i\tau
- 2\pi iu)}\theta (u).\] 
As a consequence, when $\alpha \neq 1,$ the function
 \begin{equation} 
 f_{\alpha}(u) = \frac{\theta (u - l_{\alpha})}{\theta (u)},
\label{scalartriv} 
 \end{equation} 
 where   $l_{\alpha}=   \frac{\log\alpha}{2\pi i}$, provides a
meromorphic trivialization of the factor of automorphy 
 $\xi_{\alpha}.$ The function $f_{\alpha}$ has a simple zero at
points in $\Delta - l_{\alpha} + \mathbb{Z}+\tau 
\mathbb{ Z}$ and simple poles at points in $\Delta + \mathbb{Z}+\tau 
\mathbb{ Z},$ where $\Delta = \frac{1}{2} + \frac{1}{2}\tau.$ In the
case $\alpha =1,$ nonzero meromorphic functions 
on $M$ provide trivializations. In particular, if \[ z_{1},\cdots
,z_{N}\ :\ w_{1}\cdots , w_{N} \] are $2N$ points on 
$M$ with \[z_{1}+\cdots +z_{N} -w_{1}-\cdots - w_{N}\equiv 0 \mod
(\mathbb{Z}+\tau\mathbb{Z}),\] then when considered on 
$\mathbb{C}$ the function  \[ f(u)= \prod_{i=1}^{N}\frac{\theta
(u+\Delta - z_{i})}{\theta (u+\Delta -w_{i})}\] provides 
a trivialization of $\xi_{1}$ with zeros and poles at points over the
divisor $(f)$ given on $M$ by  
\[ (f)=z_{1}+\cdots +z_{N} -w_{1}-\cdots - w_{N}.\]
\smallskip

 \par \textbf{The \ matrix\ case:} Less well known is that one can
use theta functions with matrix arguments 
 to obtain trivializations of flat matrix factors of automorphy over
$M=\mathbb{C}/\mathbb{Z}+\tau\mathbb{Z}$ 
 that is analogous to the line bundle case as follows. 

For $A$ an $r\times r$-matrix, define the matrix-valued theta
function on $\mathbb{C}$ by 
\begin{equation}
    \Theta_{A}(u) =\theta (uI_{r} + A)= \sum_{n\in \mathbb{Z}}
\exp2\pi i\left(\frac{n^{2}\tau}{2}I_{r} +n( uI_{r}+A)\right). 
    \label{mtheta} \end{equation}  The function $\Theta_{A}$ has the
following automorphic behavior:
$$
    \Theta_{A} (u+1) =
\Theta_{A}(u) \ :\ \Theta_{A}(u+\tau) =\exp{(-2\pi iA)}\exp({-\pi
i\tau - 2\pi iu)}\Theta_{A}(u). 
$$
As a result the meromorphic matrix function
\begin{equation}\overline{\Theta}_{A} =
(\theta)^{-1}\Theta_{A}\end{equation} 
is a trivialization of the rank $r$ flat factor of automorphy
$\xi_{A}$ corresponding to the representation of $\mathbb{Z} 
+ \tau\mathbb{Z}$ given by \begin{equation} \xi_{A}(1)=I_{r}\ :\
\xi_{A}(\tau ) =\exp(-2\pi iA).\end{equation}  
If one sets $V=\exp(-2\pi iA),$ then $G_{V} = \overline{\Theta}_{A}$
is a trivialization of the factor of
automorphy $\xi_{V}$ corresponding to the representation (\ref{rep}).
Note with these identifications, 
\begin{equation}G_{V}(u) =\theta^{-1}(u) \sum_{n\in \mathbb{Z}}V^{-n}
\exp2\pi i\left(\frac{n^{2}\tau}{2} +n u\right). 
    \label{niceform}\end{equation}

 The representation $\xi_{V}$ can be assumed to be
in the form 
\begin{equation} \xi_{J} (1)=I_{r}\ :\ \xi_{J} (\tau )=J,
\label{flat} 
\end{equation}
where the matrix $J$ is in Jordan canonical form.
 In particular, $J$ is a
direct sum of matrices of the form 
\begin{equation}  \label{jordan}
  J_{\alpha} = \begin{bmatrix} \alpha & 1 & & \\ & \ddots & \ddots & \\
  & & \alpha & 1 \\ & &  & \alpha \end{bmatrix}
\end{equation}
(with unspecified entries equal to 0) where $\alpha \neq 0$.

Providing trivializations of flat factors of automorphy $\xi$ on
$\mathbb{C}$ can be accomplished by providing 
trivializations of the representations of $\mathbb{Z}+\tau\mathbb{Z}$
corresponding to  irreducible factors, i.e., to
trivializations of $\xi_{J_\alpha}$ where  $J_{\alpha}$ is
the $r\times r$-matrix given by (\ref{jordan}).  
It should be noted that the factor of automorphy $\xi_{1}$
corresponds to the unique equivalence class of 
indecomposable rank $r$ vector bundles of degree $0$ on $M$ that have
a holomorphic section. These bundles 
form the basic building blocks used by Atiyah in the classic paper
\cite{Atiyah} describing vector bundles 
over an elliptic curve.  

The trivialization $ G_{J_{\alpha}}$ given by (\ref{niceform}) can be
expressed explicitly as follows. Write 
$J_{\alpha} = \alpha I + Q,$ where $Q$ is the $r\times r$
nilpotent matrix 
$$
Q= \begin{bmatrix} 0 & 1 & & \\ & \ddots & \ddots & \\
  & & 0 & 1 \\ & &  & 0 \end{bmatrix}
$$
Using the binomial expansion for $(\alpha + x)^{-n},$ one obtains 
\begin{align*}
 & J_{\alpha}^{-n} = (\alpha I + Q)^{-n} = \alpha^{-n} ( I - 
 (-\alpha^{-1} Q))^{-n} \\
 & =  \sum_{j=0}^{r-1}
\frac{(-1)^{j} n(n+1)\cdots (n+j-1)}{\alpha^{n+j} j!} Q^{j}.
 \end{align*}
Substitution of this expression for $J_{\alpha}^{-n}$ as $V^{-n}$
into (\ref{niceform}) combined with an 
interchange of the order of summation then gives
$G_{J_\alpha}(u) = (\theta (u))^{-1} \cdot \widetilde
G_{J_\alpha}(u)$ where
\begin{align} 
& \widetilde G_{J_\alpha}(u)   \notag  \\
& =  \sum_{j=0}^{r-1}\frac{(-1)^{j}}{\alpha^{j} j!} \left(\sum_{n\in
\mathbb{Z}}n(n+1)\cdots (n+j-1) \exp2\pi 
  i\left(\frac{n^{2}\tau}{2} +n(u- l_{\alpha})\right)\right)Q^{j},
\label{G}
\end{align}
where $l_{\alpha}=\frac{\log\alpha}{2\pi i}$.

 For $j=0,\ldots , r-1,$ introduce the differential operators
$L_{j},$ with $L_{0}=I$ and
 $$ 
  L_{j} = \frac{(-1)^{j}}{\alpha^{j}j!}D'(D'+1)\cdots(D'+j-1)
 $$
  where $D'=\frac{1}{2\pi i}\frac{d}{ du}$. Since
   \[L_{j}\exp(2\pi inu) = \frac{(-1)^{j} n(n+1)\cdots
(n+j-1)}{\alpha^{j} j!}\exp(2\pi inu)\] 
   one concludes that
\[ G_{J_\alpha}(u) =(\theta)^{-1}(u)\sum_{j=0}^{r-1}
(L_{j}(D')\theta)(u-l_{\alpha})\, Q^{j}\]
 or, equivalently,
\begin{equation} \label{tatiyah}
    G_{J_\alpha}(u)   = (\theta)^{-1}(u) 
  \begin{bmatrix}  \theta  & L_{1}[\theta]& L_{2}[\theta] & \cdots  &
L_{r-1}[\theta] \\ 
 & \theta  &  L_{1}[\theta]& \ddots  & \vdots  \\ 
 & & \ddots  & \ddots  &  L_{2}[\theta] \\ 
 &  &  &\ddots  & L_{1}[\theta]\\ 
 &  &  &  &\theta
\end{bmatrix}(u-l_{\alpha}). 
\end{equation}
The determinant of the trivialization $G_{J_\alpha}$ of
$\xi_{J_\alpha}$ has a zero of order $r$ at points in  
$\Delta + l_{\alpha} +\mathbb{Z} + \tau\mathbb{Z}$ and a pole of
order $r$ at points in $\Delta +\mathbb{Z}+\tau\mathbb{Z}.$ 
These zeros and poles correspond to the simple zeros and poles of the
diagonal elements at these points.
Any other trivialization of $\xi_{\alpha}$ will have the form
$G_{\alpha}K,$ where $K$ is a nondegenerate meromorphic 
$r\times r$-matrix function on $M.$

\textbf{II.\ Single\ pole\ trivializations:}  Another natural
trivialization of the rank-$r$ vector bundle $\xi_{J_\alpha}$ 
can be constructed as follows. 
For $a\in \mathbb{C},$ let 
\[
\lambda _{a}(u)=-\frac{1}{2\pi i}\frac{\theta ^{\prime
}(u-a-\frac{1}{2}-
\frac{\tau }{2})}{\theta (u-a-\frac{1}{2}-\frac{\tau }{2})},\ u\in
\mathbb{C}.
\]
The salient properties of $\lambda _{a}$
are that the poles of this function are simple poles at points in 
$a+\mathbb{Z} +\tau \mathbb{Z},$ and that 
\begin{equation}  \label{lambdaauto}
\lambda _{a}(u+1)-\lambda _{a}(u)=0, \quad \lambda _{a}(u+\tau
)-\lambda _{a}(u)=1.
\end{equation}

Note that $\lambda_{a} $ is a (multiple of) the translate
$\lambda_{a} (u) = \zeta (u-a)$ of the Weierstrass function 
\[ \zeta (u) = \frac{1}{u} + \sum_{(m,n)\neq (0,0)} \left(
\frac{1}{u-\Omega_{m,n}}+
\frac{1}{\Omega_{m,n}}+\frac{u}{\Omega_{m,n}^2}\right),\]
where $\Omega_{m,n} = m+\tau n$ (see e.g.~Problem \#5 page 279 and 
Problem \#1 page 309 in \cite{SS}).

For fixed $\alpha \neq 0$ we will use a sequence of polynomials
$\{p_{n}\}_{n\geq 0}$ with $p_{n}$
of degree $n$, $p_{0}(u)=1$,  which satisfy 
\begin{equation}  \label{pn-1}
p_{n+1}( u+1) -p_{n+1}( u) =\alpha
^{-1}p_{n}(u),\
u\in \mathbb{C}.
\end{equation}
Such a sequence of polynomials is uniquely determined if
one requires that $p_{n}(0)=0$ for $n\geq 1$, and $p_{0}=1$. These
polynomials are then given by
\begin{equation}   \label{pn-2}
p_{n} (u) = \frac{\alpha ^{-n}}{n!} (u-(n-1))\cdots (u-2)(u-1)u,\
n\geq 1.
\end{equation}
Note that the relations \eqref{lambdaauto} satisfied by $\lambda_a$
implies that the composite functions
$p_n\circ \lambda_a$ satisfy the relations
\begin{align}  
&  p_{n+1}(\lambda_a(u+1)) = p_n(\lambda_a(u)), \notag \\
&  p_{n+1}(\lambda_a(u+\tau)) = p_{n+1}(\lambda_a(u)) + \alpha^{-1}
p_n(\lambda_a(u)).
 \label{relations'}
\end{align}
If we define the $r\times r$ matrix function $P_r(\lambda)$ by
\[
P_{r}(\lambda )=
\begin{bmatrix}
p_{0}(\lambda )  & p_{1}(\lambda ) & p_{2}(\lambda ) & \cdots  &
p_{r-1}(\lambda) \\ 
 & p_{0}(\lambda ) & p_{1}(\lambda ) & \ddots  & \vdots  \\ 
 &  & \ddots  & \ddots  & p_{2}(\lambda ) \\ 
 &  &  &  p_{0}(\lambda ) &  p_{1}(\lambda ) \\ 
 &  &  &  &  p_{0}(\lambda ) \end{bmatrix},
\]
then,  for $a\in \mathbb{C} $ and $\alpha \neq 0$ fixed, it is easily
seen from the relations \eqref{relations'} that
 the meromorphic matrix function
\begin{equation}
G_{r}(u) = P_{r}(\lambda_{a}(u)) \label{basic}
\end{equation}
satisfies
\[
G_r(u+1) = G_r(u), \quad G_r(u+\tau) = \alpha^{-1} J_\alpha G_r(u),
\]
and hence, more generally,
 \[  
 G_{r}( u+m+n\tau ) = \alpha^{-n}\xi_{J_\alpha} (m+n\tau )G_{r}(u).
 \]
Thus any trivialization $F_r$ of $\xi_{J_\alpha}$ will have the form
\begin{equation}
F_r = G_r S_r \label{general} 
\end{equation}
where $S_r$ is a nondegenerate $r\times r$-meromorphic matrix
function satisfying
\begin{equation}
S_{r}(u+m+n\tau) = \alpha^{n} S_{r}(u). \label{right}
\end{equation}

When $\alpha = 1$, the matrix function $G_r$ given by (\ref{basic})
already provides a trivialization of $\xi_{J_1}.$ 
When $\alpha \neq 1$ a trivialization of $\xi_{J_\alpha}$ can be
given by $F_r$ of the form (\ref{general}), where $S_{r}$ 
is a diagonal meromorphic matrix function on $\mathbb C$ with
diagonal entries $s_{11},\ldots, s_{rr}$  
satisfying $s_{ii}(u+m+n\tau ) = \alpha^n s_{ii}(u),\ i=1,\ldots, r.$
These trivializations of $\xi_{J_\alpha}$ 
(when $r>2$) have high order poles at points in $a+\mathbb{Z} +
\tau\mathbb{Z}$. 

\section{Meromorphic trivializations with simple null-pole
structure when $g=1$}  \label{S:simple}

As mentioned in the Introduction, one of the goals of the present
paper is to find trivializations of flat factors of 
automorphy on $\widehat{M}$ which have simple null-pole structure. 
Throughout this section we again assume that $M$ is an elliptic curve
(i.e., $M$ has genus 1).

Suppose that the meromorphic matrix function $F$ is a trivialization
of the flat factor of automorphy $\xi$.
The condition that $F$ have {\em simple null-pole structure} must be
defined precisely. Let  
$z_{1},\ldots ,z_{N}:w_{1},\ldots ,w_{N}$ be $2N$ distinct
points in ``the'' fundamental domain for 
$\mathbb{C}/({\mathbb Z}+\tau{\mathbb Z})$. Let 
$\mathbf{x}_{1},\ldots ,\mathbf{x}_{N}:\mathbf{y}_{1},\ldots
,\mathbf{y}_{N}$ be nonzero (column) vectors in 
${\mathbb C}^{r}$. An $r\times r$-meromorphic matrix function $F$ on
$\mathbb{C}$ is said to interpolate the 
simple null-pole data
\begin{equation}
{\mathcal D}:(z_{1},\mathbf{x}_{1}),\ldots
,(z_{N},\mathbf{x}_{N}):(w_{1},\mathbf{y}_{1}),\ldots
,(w_{N},\mathbf{y}_{N})  \label{data'}
\end{equation}
if the following conditions are satisfied:

\begin{enumerate}
\item  The null-pole divisor of $\det F$ on the fundamental domain is 
$$
(\det F)= z_{1}+\cdots +z_{N}-w_{1}-\cdots -w_{N}.
$$

\item  The only poles of any entry of $F$ are at most simple poles at
points of $w_{i}+\mathbb{Z}+\tau \mathbb{Z}$, $i=1,\ldots, N.$

\item  The matrix function $F$ is analytic at points in
$z_{i}+\mathbb{Z}+\tau 
\mathbb{ Z}$ (already a consequence of condition (2) above) and the
vector 
$\mathbf{x}_{j}$ spans the right kernel of $F$ at these points,
$j=1,\ldots, N.$

\item  The matrix function $F^{-1}$ is analytic at points in
$w_{i}+\mathbb{Z}+\tau 
\mathbb{ Z}$ and $\mathbf{y}_{i}^{\top}
$ (here $\top$ denotes transpose) spans the left kernel of $F^{-1}$, at
these points 
$i=1,\ldots ,N.$
\end{enumerate}
When  all these conditions are satisfied, it can be shown that the
poles of $F$ are all simple and occur precisely at 
the points $w_1, \dots, w_N$
with rank 1 residue at $w_i$  having left image spanned by the vector
${\mathbf y}_i^\top$ for each $i$, while the poles 
of $F^{-1}$ are all simple and occur
 precisely at the points $z_1, \dots, z_N$ with rank-1 residue at
$z_i$ having right image spanned by the vector 
 ${\mathbf x}_i$ for each $i$. For further information, see
\cite{BGR,Kats}.

If $F$ satisfies (1)--(4), then we say that $F$ has {\em simple
null-pole data structure} and that $\mathcal{D}$ 
given by \eqref{data}  is the  {\em null-pole divisor of $F$}, or
that $F$ interpolates the data set ${\mathcal D}$.

\smallskip

We will establish the following:

\begin{thm}
Given a rank $r$ flat factor of automorphy $\xi $ over the elliptic
curve 
$M=\mathbb{C}/(\mathbb{Z}+\tau \mathbb{Z})$, there exist a simple
null-pole data set
$\mathcal{D}$ of the form (\ref{data'}) and a trivialization $F$ of
$\xi$ such that $F$  interpolates ${\mathcal D}$. 
In general, the size  $N=N_{r}$ of ${\mathcal D}$  can be taken to 
satisfy $N_{r} < 2r.$  Moreover, if 
$\xi $ does not have a direct summand equivalent to a representation
$\xi _{J_1}$,  then it is possible to take $N_{r}=r.$ 
\end{thm}

This theorem follows directly from the next proposition which is
established via a constructive inductive proof based 
on the rank $r$.

\begin{prop}
For $\alpha \neq 0$ there exist a simple null-pole data set
$\mathcal{D}$ of the
form (\ref{data}) and a trivialization $F_{r}$ of the form
(\ref{general}) of $\xi_{J^{(r)}_\alpha}$ which interpolates 
${\mathcal D}$.  If $\alpha =1$ , $N = N_{r}$ can be taken in the
form $N_{r}=2(r-1)$. If $\alpha \neq 1,$ then it is 
possible to take $N_{r}=r.$
\end{prop}

\begin{proof}
Consider first the case $r=1$.  If $\alpha = 1$, the constant
function $f(p) = 1$ is a trivialization of $\xi_1$.  
If $\alpha \ne 1$, as shown in (\ref{scalartriv}), one can use theta
functions to obtain a scalar function having 
one simple zero and one simple pole with factor of automorphy equal
to $\alpha$. 

Inductively
suppose now that the proposition has been established for
$\xi_{J_\alpha^{(r)}}$ for the case where $J^{(r)}_\alpha$ 
is the $r \times r$ matrix of the form  \eqref{jordan}.  In more
detail,  we suppose that it has been shown that there 
is a trivialization $F_{r}  = G_r S_r$  as in  \eqref{general} 
with $G_r$ of the form \eqref{basic} and $S_r$ satisfying
\eqref{right} such that $F_r$
interpolates a simple null-pole data set  $\mathcal{D}_{r}$ as in
\eqref{data} (so $F = F_r$ satisfies conditions (1)--(4)) 
with $N=N_{r}=r$ if $\alpha\neq 1$, and $N=N_{r}=2(r-1)$, if $\alpha =
1.$  It will be further assumed that the interpolation points 
$z_1,\ldots ,z_{N_{r}} : w_1,\ldots, w_{N_{r}}$ are different from
$a$. 
For convenience, we also assume that  $a=0.$.

Write $S=S_{r}$ in the form 
\[
S_{r}= \begin{bmatrix} \mathbf{s}_{r} \\  \mathbf{s}_{r-1} \\  \vdots
\\  \mathbf{s}_{1} \end{bmatrix} ,
\]
where $\mathbf{s}_{1},\ldots ,\mathbf{s}_{r-1},\mathbf{s}_{r}$ are
$r$-dimensional row vector functions. 

 We wish to construct a trivialization $F_{r+1}$ of
$\xi_{J^{(r+1)}_\alpha}$ of the form 
 \begin{equation}   \label{Fr+1}
   F_{r+1} = G_{r+1} S_{r+1}
 \end{equation}
 also as in \eqref{general}.  Thus  $G_{r+1}$ should have the form 
 \begin{equation}   \label{Gr+1}
  G_{r+1}(u) = \begin{bmatrix} 1 & {\mathbf p}(u) \\ 0 & G_r(u)
\end{bmatrix}
  \end{equation}
   with  $\mathbf{p}$ equal to the $r$-dimensional row vector function
$$
  \mathbf{p}(u)=  \begin{bmatrix}  p_{1}(\lambda_{0}(u)) &
p_{2}(\lambda_{0}(u)) & \ldots  & p_{r}(\lambda_{0}(u)) \end{bmatrix} 
 $$
where   $p_j$ ($j = 1, 2, \dots$) are the polynomials as in
\eqref{pn-1} and \eqref{pn-2}, and $S_{r+1}(u)$ should be an $(r+1) 
\times (r+1)$ matrix function satisfying \eqref{right} which we
assume to have the form
\begin{equation}  \label{Sr+1}
   S_{r+1} = \begin{bmatrix} s_0 & {\mathbf s}_{r+1}  \\ 0 & S_r
\end{bmatrix}.
\end{equation}
From \eqref{Fr+1} combined with \eqref{Gr+1} and \eqref{Sr+1} we get
\begin{equation}   \label{r+1}
F_{r+1}=G_{r+1} \begin{bmatrix} s_{0} & \mathbf{s}_{r+1} \\ 0 & S_{r}
\end{bmatrix}
=  \begin{bmatrix} s_{0} & \mathbf{s}_{r+1}+\mathbf{p}S_{r} \\  0
&G_{r} S_{r} \end{bmatrix}
\end{equation} 
where here we use the fact that $p_0(u) = 1$.

By assumption $F_r = G_r S_r$ trivializes $\xi^{(r)}_\alpha$ and has
a simple matrix null-pole divisor supported on $z_1 +
\cdots + z_{N_{r}} - w_1 - \cdots - w_{N_{r}}$. The goal is to choose
$s_0$ and ${\mathbf s}_{r+1}$ so that 
\begin{enumerate}
\item[(i)] 
$F_{r+1}$ trivializes $\xi_{J^{(r+1)}_\alpha}$,  and 
\item[(ii)] $F_{r+1}$ has a simple matrix null-pole divisor with one
additional pole and zero, if $\alpha \ne 1$ and two
additional poles and zeros, if $\alpha = 1.$
\end{enumerate}

To achieve condition (i), by the analysis in Section 2 we need only
guarantee that $S_{r+1}$ satisfies \eqref{right}.  
As $S_r$ satisfies \eqref{right} by the induction hypothesis, this
condition reduces to the two conditions
 \begin{equation}
 s_{0}(u+1)=s_{0}(u),\ s_{0}(u+\tau )=\alpha s_{0}(u). \label{scalar} 
 \end{equation}
 and
 \begin{equation} \label{vector}
\mathbf{s}_{r+1}(u+1)=\mathbf{s}_{r+1}(u),\ \
\mathbf{s}_{r+1}(u+\tau)=\alpha\mathbf{s}_{r+1}(u)
\end{equation}
 
 We therefore take the scalar function $s_{0}$ to be a meromorphic
function on $\mathbb{C}$ satisfying \eqref{scalar}.
 To help meet requirement (ii), we arrange that the 
 null-pole divisor $(s_{0})$ of $s_{0}$ is as small as possible; in
case $\alpha = 1$, we can arrange that
$$
(s_{0})=\zeta _{1}+\zeta _{2}-\pi _{1}-\pi _{2}
$$
where the points $\zeta_{1},\zeta_{2},\pi_{1},\pi_{2}$  are chosen
disjoint from points in the data $\mathcal{D}_{r}$ and $a=0.$  
In the case
where $\alpha \neq 1$, we shall take the divisor  $(s_{0})$ to be of
the simpler form 
$$
(s_{0})=\zeta _{1}-\pi _{1}.
$$

The vector function $\mathbf{s}_{r+1}$ is an $r$
dimensional meromorphic row vector function on $\mathbb{C}$ that
satisfies 
\begin{equation} \label{vector'}
\mathbf{s}_{r+1}(u+1)=\mathbf{s}_{r+1}(u),\ \
\mathbf{s}_{r+1}(u+\tau)=\alpha\mathbf{s}_{r+1}(u)
\end{equation}
 that remains to be chosen.

From the automorphic properties of $F_{r+1}$ (as also can be seen
directly from \eqref{scalar} and \eqref{vector}) we have
\begin{equation} \label{key1}
 \mathbf{s}_{r+1}(u+1)+\mathbf{p}(u+1)S_{r}(u+1) =
\mathbf{s}_{r+1}(u)+\mathbf{p}(u)S_{r}(u)
\end{equation} and 
\begin{equation} \label{key}
 \mathbf{s}_{r+1}(u+\tau )+\mathbf{p}(u+\tau)S_{r}(u+\tau ) = \alpha
\mathbf{s}_{r+1}(u)+\alpha \mathbf{p(u)}S_{r}(u) +
 \overline{\mathbf{p}}(u)S_{r}(u).  \end{equation}
  where
  \[ 
 \overline{\mathbf{p}}(u) =  \begin{bmatrix}  p_{0}(\lambda_{0}(u)) &
p_{1}(\lambda_{0}(u)) &  \ldots  &
 p_{r-1}(\lambda_{0}(u)) \end{bmatrix}. 
 \]  
For later reference, note that the term $\overline{\mathbf{p}}S_{r}$
appearing in equation (\ref{key}) is 
the first row of $F_{r}=G_{r}S_{r}.$ In particular, the induction
assumption guarantees that this term does not 
have a pole at $u=0.$

To complete the achievement of requirement (ii),  the entries
$s_{r+1,j}$, ($j=1\hdots, r$) of $\mathbf{s}_{r+1}$
must be chosen so that the null-pole divisor of $F_{r+1}$ is simple.

For $j\ge 1$ the $(1, j+1)$-entry in $F_{r+1}$ is 
\begin{equation}  \label{entry} 
 s_{r+1,j} + ( p_1 \circ \lambda_0)  [S_{r}]_{1,j} +\cdots + ( p_r
\circ \lambda_0) [S_{r}]_{r,j}. 
\end{equation}
We first choose $s_{r+1,j}$ to remove any pole of $( p_1 \circ
\lambda_0)  [S_{r}]_{1,j} +\cdots + ( p_r \circ \lambda_0) 
[S_{r}]_{r,j}$ at $a=0$.
In the case $\alpha \neq 1,$ the removal of poles at $a=0$  can be
accomplished by using theta functions. In more detail, 
given a polynomial $p$ with 
$p(0)=0$, there exists a meromorphic function $s=s(u)$ with pole only
at $0$ satisfying 
\begin{equation}  \label{remove pole} 
s(u+m+\tau n)=\alpha ^{n}s(u),\ p(1/u)-s(u)\text{ is analytic at} \
u=0.
\end{equation}
This statement  can be
established as follows: Let $\Delta =\frac{1}{2}+\frac{1}{2}\tau .$
First
assume $\alpha \neq 1$ and let $\mu _{k}=\frac{1}{2\pi ik} \log
\alpha +\Delta $. For an appropriate choice of $C_{k}$ the 
function 
\[
q_{k}(u)=C_{k}\left( \frac{\theta (u-\mu _{k})}{\theta (u-\Delta
)}\right)
^{k} 
\]
satisfies $q_{k}(u+m+n\tau )=\alpha ^{n}q_{k}(u)$ and has a Laurent
expansion about $u=0$ which begins with the term $\frac{1}{u^{k}}$.
By
choosing appropriate linear combinations of the functions
$q_{k},...,q_{1}$,
for $k\geq 1,$ one can construct functions $p_{k}$ satisfying
$p_{k}(u+m+n\tau
	)=\alpha ^{n}p_{k}(u)$ with principal part $\frac{1}{u^{k}}$ at
$u=0$. Therefore, in the case $\alpha \neq 1,$ 
	the existence of a function $s$ with  property (\ref{remove pole})
follows.  

When $\alpha = 1,$ given a polynomial $p=p(u)$
with $p(0)=0,$ there exists an elliptic function $s$ such that the
principal
part of  $s$ at $u=0$ is $p( \frac{1}{u} ).$
This is easily established using the Weierstrass functions 
$\varsigma,\wp ,\wp^{\prime }, \dots$ associated with the lattice
$\mathbb{Z}+\tau
\mathbb{Z}$. Thus the existence of a function $s$ 
satisfying (\ref{remove pole}) follows.

The existence of $s$ satisfying (\ref{remove pole}) allows one to
choose $s_{r+1,j}$ so that the entries (\ref{entry})
do not have poles at $a=0.$ It follows from equations (\ref{key1})
and (\ref{key}) that $F_{r+1}$ does not have poles 
at points in $0+\mathbb{Z}+\tau\mathbb{Z}.$

At this point the (1,2)-entry \[-s_{0}^{-1}(\mathbf{s}_{r+1} +
\mathbf{p}S_{r})F_{r}^{-1}\] of $F_{r+1}^{-1}$ may have 
poles at points in $\{w_{1},\ldots,w_{N}\} +
\mathbb{Z}+\tau\mathbb{Z}.$ As above, by suitably modifying the
entries in 
$\mathbf{s}_{r+1},$ one can remove any of these poles at fixed
representatives of the points, while maintaining (\ref{vector}) . 
The identities (\ref{key1}) and (\ref{key}) along with the fact that
$\overline{\mathbf{p}} S_{r}F_{r}^{-1}$ is the row of  the 
identity matrix,  imply that the first row of $F_{r}^{-1}$ does not
have poles at any point in  
$\{w_{1},\ldots,w_{N}\} + \mathbb{Z}+\tau\mathbb{Z}.$  It follows
that $F_{r+1}$ has the correct pole structure at 
\[  \{w_{1},...,w_{N},\pi _{1},\pi _{2}\}+\mathbb{Z}+\tau \mathbb{Z}
\]
when $\alpha = 1$ and at  \[\{w_{1},\ldots ,w_{N},\pi
_{1}\}+\mathbb{Z}+\tau \mathbb{Z},\]
when $\alpha \neq 1.$

There remains to check that $F_{r+1}$ and $F_{r+1}^{-1}$ have the
correct zero struc\-ture. The zero divisor of the 
determinant of $F_{r+1}$ is $z_{1}+\cdots+z_{N}+\zeta_{1}$ in the
case $\alpha\neq 1$ and 
$z_{1}+\cdots+z_{N}+\zeta_{1}+\zeta_{2}$ in the case $\alpha =1.$ The
form 
$$
F_{r+1}^{-1} = \begin{bmatrix} 
s_{0}^{-1} & -s_{0}^{-1} (\mathbf{s}_{r+1}+\mathbf{p}S_{r})F_{r}^{-1} \\ 
0 & F_{r}^{-1}
\end{bmatrix}
$$
allows one to conclude that the entries of $F_{r+1}^{-1}$ have at
most simple poles and only at the zeros of the determinant of 
$F_{r+1}.$
At $ z_{j},\  j=1,\ldots, N$ the right kernel of $F_{r+1}(z_{j})$ is
spanned by
 $$
 \begin{bmatrix}
 -s_{0}^{-1}(z_{j}) (
\mathbf{s}_{r+1}+\mathbf{p}S_{r})(z_{j})(\mathbf{x}_{j}) \\ 
\mathbf{x}_{j} \end{bmatrix}. 
$$
 The zero divisor of the determinant of $F_{r+1}^{-1}$ is
$w_{1}+\cdots+w_{N}+\pi_{1}$ in the case $\alpha\neq 1$ and 
 $w_{1}+\cdots+w_{N}+\pi_{1}+\pi_{2}$ in the case $\alpha =1.$  At
$w_{j},\  j=1,\ldots, N$ the right kernel of 
 $F_{r+1}^{-1}(w_{j})$ is spanned by $[0,\mathbf{y}_{j}^{\top}].$ At
$\pi_{i}$  the right kernel is 
 spanned by the $r+1$-dimensional row vector
$$
\begin{bmatrix}  1   
    & -(s_{0}^{-1}(\mathbf{s}_{r+1}+\mathbf{p}S_{r})F_{r}^{-1})(\pi
_{i})F_{r}(\pi_{i})
    \end{bmatrix} 
$$
for $i = 1,2$.  The proof is complete.
\end{proof}

\section{Trivialization of factors of automorphy via matrix-divisor 
constructions}   \label{S:divisors}

\subsection{Right matrix null/pole divisors, matrix-divisor spaces, 
vector bundles, and factors of automorphy}  \label{S:4.1}

 A general theory of null-pole divisors of meromorphic matrix
functions, extending the case of simple null-pole 
 structure described in the preceding section, is presented in
\cite{BGR}. This general theory can be used to elucidate 
 the concrete results on trivializations described above, and 
 furthermore applies equally well to the higher genus case. Suppose $F$
is a meromorphic $r\times r$-matrix function 
 defined in a neighborhood of a point $q_{0}$ on a Riemann surface
with $\det F \not\equiv 0$.  In local coordinates 
 $(z,U)$ near $q_{0}$ with $z(q_{0})=0$ one introduces a local
(right) null-pole triple $\Upsilon$, that captures 
 the null-pole behavior of $F.$ This null-pole triple has the form 
 \begin{equation}  \label{triple}
  \Upsilon = ((B_{\zeta}, A_{\zeta}), (A_{\pi},C_{\pi}), S). 
  \end{equation} 
  In this triple, the pair of matrices, $(A_{\pi},C_{\pi}),$ where
$A_{\pi}$ is $n_{\pi}\times n_{\pi}$ and $C_{\pi}$ 
  is $n_{\pi}\times  r$  captures the pole behavior of $F$ at $q_{0}$
in the sense that for some matrix $\widetilde{B}$ 
  the matrix function
 \begin{equation}   \label{pole}
  F(q) - \widetilde{B} (z(q)I-A_{\pi})^{-1}C_{\pi}
   \end{equation}
 is analytic at $q_{0}$ and the matrix size $n_\pi$ is as small as
possible so that (\ref{pole}) holds. The pair $(A_{\pi}, C_{\pi})$ is 
called a {\em left pole pair} for $F$ since it is a {\em left null 
pair} for $F^{-1}$ in the sense that $(z(q) I - A_{\pi})^{-1} C_{\pi} 
F(q)^{-1}$ has analytic continuation to $q_{0}$ (the zero of 
$F^{-1}$, i.e., pole of $F$, is canceled by the pole of $(z(q) I - 
A_{pi})^{-1} C_{\pi}$ at $q_{0}$).
In a similar
 manner, the pair $(B_{\zeta}, A_{\zeta}),$ where $A_{\zeta}$ is
$n_{\zeta}\times n_{\zeta}$ and $B_{\zeta}$ is 
 $r\times n_{\zeta}$,  captures the zero behavior of $F$ at $q_{0}$ in
the sense that for some matrix $\widetilde{C}$, 
  \begin{equation}    \label{zero}
   F^{-1}(q) -  B_\zeta (z(q)I-A_{\zeta})^{-1} \widetilde C
    \end{equation}
 is analytic at $q_{0}$  with the matrix size $n_\zeta$ again as
small as possible.  We then refer to $(B_{\zeta}, A_{\zeta})$ as a 
{\em right null pair} for $F$ since $F(q) B_{\zeta} (z(q)I - 
A_{\zeta})^{-1}$ has analytic continuation to $q_{0}$ (the zero of 
$F$ is cancelled on the right by the pole of $B_{\zeta}(z(q) I - 
A_{\zeta})^{-1}$ at $q_{0}$).

The $n_{\pi}\times n_{\zeta}$ 
 matrix $S$, called the coupling matrix, is  that solution of the
Sylvester equation
 \begin{equation}   \label{Syl}
  A_{\pi}S-SA_{\zeta} = C_{\pi}B_{\zeta}
  \end{equation}
 which encodes the additional information needed to completely
specify the ${\mathcal O}_{q_{0}}$-row module 
 ${\mathcal O}^{1 \times r}_{q_{0}}  \cdot F$ (where ${\mathcal
O}_{q_{0}}$ is the space of germs of functions 
 holomorphic on a neighborhood of $q_0$).   Information equivalent to
knowledge of  the row module
 ${\mathcal O}^{1 \times r}_{q_{0}}  \cdot F$ is  knowledge of {\em
right germ} of the nondegenerate 
 meromorphic matrix function $F$ at $q_0$.  Here the {\em right germ}
of $F$ at $q_0$ is the equivalence class of 
 nondegenerate meromorphic matrix functions on a neighborhood of
$q_0$, where two such functions $F$ and $F'$ 
 are considered equivalent if there is  matrix function $H$
holomorphic and invertible on a neighborhood of $q_0$ 
 such that $F' = H F$.  
 
 The precise connection between the right null-pole triple $\Upsilon$ given as
in (\ref{triple}) and the module  
${\mathcal O}^{1 \times r}_{q_{0}} \cdot F$ is as follows (see \cite{BGR} as 
well as \cite{BRlocal, BRglobal} for the genus 0 case, \cite{Kats} 
for an expository treatment of the simple-multiplicity genus 0 case, 
and \cite{BCV} for the Riemann surface case). 
Suppose the local null-pole triple 
$\Upsilon$ of $F$ at $q_{0}\in M_{0}$ has the form (\ref{triple})
given in terms of the local coordinate $z$ at $q_{0}$.    Then
\begin{equation}   \label{local}
    \mathcal{O}^{1\times r}_{q_{0}}\cdot F = \cS(\Upsilon, q_{0},z)
\end{equation}
where we set
\begin{align}
\cS(\Upsilon, q_{0},z) = & 
\{ \mathbf{x}( z(q)I-A_{\pi}))^{-1}C_{\pi} +\mathbf{h}(q) \colon 
\mathbf{x}\in\mathbb{C}^{1\times n_{\pi}},\
\mathbf{h}\in \mathcal{O}^{1\times r}_{q_{0}} \notag \\ 
& \quad \text{such that } \mathbf{x}S =
\text{res}_{q_{0}}[\mathbf{h}(q)
B_{\zeta}(z(q)I-A_{\zeta})^{-1}]\}.
\label{nullpolesubspace} 
\end{align} 
The set $\cS(\Upsilon, q_{0},z)$ defined by [\ref{nullpolesubspace}]  
will be referred to as the singular subspace of the $0$-admissible
Sylvester 
data set $\Upsilon$ at the point $q _{0}$ with respect to local 
coordinate $z$. For simplicity we shall assume that it is understood 
that there is choice of local coordinate understood and write simply
$\cS(\Upsilon, q_{0})$.  It should be noted that the residue
appearing in this last equation is a matrix residue. 
We also point out that the role of the Sylvester equation \eqref{Syl} 
is to guarantee that the set $\cS(\Upsilon, q_{0},z)$ is indeed a 
left module over $\cO_{q_{0}})$.

 Moreover those quintuples of matrices
$\Upsilon$ \eqref{triple} which can arise as the local 
 null-pole triple for a nondegenerate  meromorphic matrix function
$F$ at a point $q_0$ are characterized as the 
 {\em $0$-admissible Sylvester data sets}, i.e., the quintuples of
matrices $\Upsilon$ \eqref{triple} such that:
 \begin{enumerate}
 \item[(a)]  $A_\pi$ is nilpotent (i.e., $\sigma(A_\pi) = \{0\}$) and
the input pair  $(A_\pi, C_\pi)$ is {\em controllable}  
 (i.e., ${\rm span}\{  {\rm Ran} A_\pi^j C_\pi \colon 0 \le j \le
n_\pi-1\} = {\mathbb C}^{n_\pi}$),
 
 \item[(b)]  $A_\zeta$ is nilpotent (i.e., $\sigma(A_\zeta) = \{0\}$)
and  the output pair $(B_\zeta, A_\zeta)$ is 
 {\em observable} (i.e.,
 $\bigcap_{j=0}^{n_\zeta - 1} {\rm Ker} B_{\zeta} A_\zeta^j = \{0\}$),
 
 \item[(c)] $S$ satisfies the Sylvester equation \eqref{Syl}.
 \end{enumerate}
 Different local null-pole triples for $F$ are related by a pair of
similarities of $A_{\pi}$ and $A_{\zeta}$. More 
 specifically, if $U$ and $V$ are invertible matrices of appropriate
respective sizes, then 
 \begin{equation} \label{sim}
 \widetilde{\Upsilon} = ((B_{\zeta}U, U^{-1}A_{\zeta}U),
(V^{-1}A_{\pi}V,V^{-1}C_{\pi}), V^{-1}SU) 
  \end{equation}
  is also a local  null-pole triple for $F$ and any other local
null-pole triple (with respect to the same local coordinate $z$ at 
$q_{0}$) has this form.  The $0$-admissible 
  Sylvester data sets $\Upsilon$ and $\widetilde{\Upsilon}$ given by
\eqref{triple} and \eqref{sim} are then said 
  to be similar Sylvester data sets.
 
\begin{remark} \label{R:liftdata}
    {\em In the discussion that follows  we will work
mainly with relatively automorphic meromorphic matrix functions 
on the universal cover $\widehat M$ of the Riemann surfaces $M$.
If we fix a choice 
  of coordinate for a sufficiently small neighborhood at a point
$q_0$ in a Fundamental Domain $M_{0}\subset\widehat{M}$ 
  and then use a deck transformation $T$ to  lift this coordinate to
a coordinate for a neighborhood of the point $T(q_0)$ 
  lying above $q_0$,  the relative-automorphy property of $F$
guarantees that the local null-pole triple at $T(q_0)$ for $F$ 
  is exactly the same as the local null-pole triple at $q_0$ for $F$
(since the invertible left factor of automorphy $\xi(T, u)$
can be absorbed into the free invertible matrix left factor  $H$ in
the definition of right matrix germ).  In summary, this compatible 
choice of local coordinates leads to identical null-pole triples for 
points over the same base point in the fundamental domain $M_{0}$.}
\end{remark}

 The reader should be aware however that our definitions here
correspond to {\em right null-pole triples} 
 (describing the {\em row} module
 ${\mathcal O}^{1 \times r}_{q_{0}} \cdot F$ or equivalently the {\em right}
matrix germ $\{ H \cdot F \colon H^{\pm 1} 
 \in {\mathcal O}^{r \times r}_{q_{0}} \}$) whereas the main focus in
\cite{BGR} is on {\em left null-pole triples} 
 which describe the {\em column} module $F \cdot {\mathcal
O}_{q_{0}}^{r \times 1}$ or equivalently  the 
 {\em  left} matrix germ $\{ F \cdot H \colon H^{\pm 1} \in {\mathcal
O}^{r \times r}_{q_{0}}\}$.

We now illustrate the characterization \eqref{local} with a simple 
example.

\begin{example} \label{E:1} {\em
    Let $f$ be the scalar function defined in
local coordinates near zero by $f(u)=u^2$. Note that 
\begin{equation}   \label{local'}
\cO_{0} \cdot 
f = \{ h(u) = \sum_{j=0}^{\infty} h_{j} u^{j} \text{ holomorphic at } 
0 \colon h_{0} = h_{1} = 0\}.
\end{equation}
On the other hand,
any local null-pole triple for $f$ at $u=0$ consists only of a zero 
pair, i.e., $n_{\pi} = 0$, and the condition  in \eqref{local}
collapses 
to
\begin{equation}  \label{local''}
    {\rm res}_{0}\,  [h(u) B_{\zeta}(uI - A_{\zeta})^{-1} ] = 0.
\end{equation}
One choice of null pair for $f$ at $u= 0$ is the pair $ (B_{\zeta},
A_{\zeta}) = 
\left(  \left[ \begin{smallmatrix} 1 & 0 \end{smallmatrix} \right] , \left[ 
\begin{smallmatrix}  0 & 1 \\ 0 & 0 \end{smallmatrix} \right]
\right)$,
where  equation
(\ref{zero}) holds with $\widetilde{C} = \left[ \begin{smallmatrix} 0
\\ 1 
\end{smallmatrix} \right]$. If $h(u) = h_{0}+h_{1}u +\cdots$ is
analytic 
at zero, then 
\begin{align*}
{\rm res}_{0} \left\{ h (u) \left[ \begin{smallmatrix} 1 & 0 
\end{smallmatrix} \right]   \left( \left[ \begin{smallmatrix} u & 0 \\ 0 & u 
\end{smallmatrix} \right] -  \left[ \begin{smallmatrix} 0 & 1 \\ 0 &
0 
\end{smallmatrix} \right] \right)^{-1} \right\}  &  =
\text{res}_{0} \left\{ h(u) \left[ \begin{smallmatrix} 1 & 0 \end{smallmatrix}
\right]
\left[ \begin{smallmatrix} u^{-1} & 
u^{-2} \\ 0 & u^{-1} \end{smallmatrix} \right]\right\} \\
& = \begin{bmatrix} h_{0} & h_{1}\end{bmatrix}.
\end{align*}
Combining \eqref{local'} with \eqref{local''}, we arrive at a 
verification of the characterization \eqref{local} for this case.}
\end{example}

\begin{remark} \label{R:1}  {\em
We define the {\em adjoint} $\Upsilon^{*}$ of a $0$-admissible
Sylvester data set
$$
\Upsilon = ((B_{\zeta}, A_{\zeta}), (A_{\pi},C_{\pi}), S)
$$
as the Sylvester data set  
\[
\Upsilon^{*} = ((C_{\pi}^{\top}, A_{\pi}^{\top}),(A_{\zeta}^{\top},
B_{\zeta}^{\top}),- S^{\top})
\]  
where $X^{\top}$ denotes the {\em transpose} of the matrix $X$.
If the meromorphic matrix $F$ locally interpolates $\Upsilon$ at
$q_{0},$ then $(F^{\top})^{-1}$ locally interpolates the divisor 
$\Upsilon^{*}$ at $q_0$ (with respect to the same local coordinate).}
\end{remark}

Let us introduce the notation 
\begin{equation} \label{data}
\mathcal{D} =\{ (\Upsilon_{q_{0}},q_{0}) \colon q_{0} \in M_{0} \}
\end{equation} 
for a collection of  $0$-admissible Sylvester
data sets 
\begin{equation}   \label{localdata}
\Upsilon_{q_{0}}  = ((B_{\zeta_{q_{0}}}, A_{\zeta_{q_{0}}}), 
(A_{\pi_{q_{0}}},C_{\pi_{q_{0}}}), S_{q_{0}})
\end{equation}
tagged to each point 
$q_{0} \in M_{0}$.  Here it is understood that for each point $q_{0}$
there is also specified a choice 
$z$ of local coordinate at $q_{0}$ in order for 
the formula \eqref{nullpolesubspace} for the singular subspace 
$\cS(\Upsilon, q_{0},z_{0})$ to be well defined.
For definiteness, the sizes of the matrices in $\Upsilon_{u_{0}}$
are specified as follows:  $A_{\pi_{u_{0}}}$ is $n_{\pi_{u_{0}}}\times
n_{\pi_{u_{0}}}$ and $C_{\pi_{u_{0}}}$ is $n_{\pi_{u_{0}}}\times  r,$ 
$A_{\zeta_{u_{0}}}$ is $n_{\zeta_{u_{0}}}\times n_{\zeta_{u_{0}}}$ and
$B_{\zeta_{u_{0}}}$ is $r\times n_{\zeta_{u_{0}}},$  and $S_{u_{0}}$
is a 
$n_{\pi_{u_{0}}}\times n_{\zeta_{u_{0}}}$ matrix for $i=1,\ldots,k$. 
We shall also impose the restriction that the local data 
$(\Upsilon_{u_{0}}, u_{0})$ are trivial (i.e., both 
$n_{\pi_{u_{0}}}= 0$ and $n_{\zeta_{u_{0}}} = 0$) for all but
finitely many 
points $u_{0} = u_{1}, \dots, u_{k}$  in $M_{0}$. 
In the sequel a data set $\mathcal{D}$
as in (\ref{data}) subject to this finite-support restriction
will be referred to as a {\em (right matrix) null-pole divisor}.
The  integer 
$$
\deg(\mathcal{D}) = \sum_{u_{0} \in M_{0}} (n_{\zeta_{u_{0}}} -
n_{\pi_{u_{0}}})
$$
is called the {\em degree} of the divisor $\mathcal{D}.$ In
what follows, it will often be assumed that the degree of 
$\mathcal{D}$ is zero. Thus in this case $N:= \sum_{ u_{0} \in 
M_{0}}n_{\zeta_{u_{0}}} =
\sum_{ u_{0} \in M_{0}}n_{\pi_{u_{0}}}.$ 

For later purposes, it is useful to have a partitioning of the index 
set $\{1, \dots, k \}$ into three types:
\begin{align}
   {\rm I} & = \{i \colon n_{\pi_{u_{i}}} > 0 \text{ while } 
   n_{\zeta_{u_{i}}} = 0 \}, \notag \\
{\rm II} &  = \{ i \colon \text{ both } n_{\pi_{u_{i}}} > 0 \text{
and } 
   n_{\zeta_{u_{i}}} > 0 \},\notag \\
   {\rm III} &   = \{ i \colon n_{\pi_{u_{i}}} = 0 \text{ while } 
   n_{\zeta_{u_{i}}} > 0 \}.
   \label{data''}
   \end{align}
Without loss of generality we may assume that 
\begin{align*}
 {\rm I} & = \{ 1, \dots, n_{\infty}\},  \\
{\rm II}  &  = \{ n_{\infty}+1, \dots, n_{\infty}+ n_{c}\},   \\
{\rm  III} & = \{ n_{\infty} + n_{c}+1, \dots, n_{\infty} + n_{c} +
n_{0}\}.
\end{align*}
Then $k = n_{\infty} + n_{c} + n_{0}$ where $n_{\infty}$ counts the
number 
of points where there is  a pole but no zero,  $n_{0}$ counts 
the number of points where there is a zero but no pole, $n_{Z} =
n_{0} 
+ n_{c}$ counts the number of points where there is a zero,  and
$n_{P} = n_{\infty} + n_{c}$ counts the number of points where there
is a 
pole and $n_{c}$ counts the number of points  in $M_{0}$ where there
there is both a zero and a pole.

One can lift the data to $\widehat{\mathcal{D}}$ on the universal
cover $\widehat{M}$ (as in Remark \ref{R:liftdata}) and ask whether 
there are automorphic 
meromorphic matrix functions or relatively automorphic meromorphic matrix
functions with respect to a flat factor of automorphy 
interpolating the data $\widehat{\mathcal{D}}$ on $\widehat{M}.$ Such
questions were addressed in \cite{BCV,BV} and 
will be discussed below.

Given a null-pole divisor \eqref{data}, we associate a linear
matrix-divisor space 
$\cL_{*}(\cD)$ of meromorphic $(1 \times r)$-row vector functions on
$M$ 
by
\begin{equation}  \label{RRspace}
    \cL_{*}(\cD) = \{ f \in \cM(M)^{1 \times r} \colon f(q) \in 
    \cS(\Upsilon_{q_{0}}, q_{0}) \text{ for all } q_{0} \in M \}
\end{equation}
or the more general sheaf version:  for $U$ equal to any open subset 
of $M$, define $\cL_{*}(\cD)|U$ by
$$
\cL_{*}(\cD)|_{U} = \{ f \in \cM(U)^{1 \times r} \colon f(q) \in 
\cS(\Upsilon_{q_{0}},q_{0}) \text{ for all } q_{0} \in U\}.
$$
The matrix-divisor space $\cL_{*}(\cD)$ is a matrix analogue of the
space $L(-D)$ 
for $D = \sum_{p \in M_{0}} n_{p} p$ a classical scalar divisor
(formal sum of points in $M_{0}$ with multiplicities $n_{q_{0}}$ such 
that  $n_{q_{0}} = 0$ for all but finitely many $q_{0}$) given by
$$
L(-D) = \{ f \in \cM(M): (f) \ge D \}
$$
where $(f)$ denotes the pole-zero divisor of $f$.
Analogous to what is done for the scalar-valued case (see 
\cite[Section 29.11]{Forster}) where one associates a holomorphic
line bundle with
a scalar divisor, one can associate a holomorphic 
vector bundle $E_{\cD}$ with the right matrix null-pole divisor $\cD$
in such a way that the 
space of holomorphic sections of the adjoint bundle $E_{\cD}^{*}$ is
isomorphic to the 
space of meromorphic functions $\cL_{*}(\cD)$ \eqref{RRspace} as
follows.

Given a data set $\cD$ as in \eqref{data}, by the results from 
\cite{BGR} we can find a open covering
${\mathfrak U} = \{ U_{\alpha}\}_{\alpha \in A}$ of $M$ 
and invertible $r \times r$-matrix-valued meromorphic 
functions $L_{\alpha}$ on $U_{\alpha}$ so that $L_{\alpha}$ solves 
the zero-pole interpolation problem for $\cD$ restricted to 
$U_{\alpha}$:
$$ \cO^{1 \times r}_{q_{0}}\cdot  L_{\alpha} = \cS(\Upsilon_{u_{0}},
u_{0}) 
\text{ for all } u_{0} \in U_{\alpha}.
$$ 
The collection of transition functions
 $$
   \Phi_{\alpha, \beta} = L_{\alpha} L_{\beta}^{-1} \in 
   \cO^{r \times r}_{U_{\alpha} \cap U_{\beta}}
$$
 defines the equivalence class of an $r$-dimensional holomorphic
vector bundle $E_{\cD}^*$ on $M$
 (with sections locally identified with row vector functions).
In view of the connection \eqref{local} between 
interpolants and null-pole subspaces, we see that:  {\em if $U$ is
any 
open subset of $M$, then 
the row-vector function $f$ is in $\cL_{*}(\cD)|_{U}$ if and only if}
$$
 f|_{U \cap U_{\alpha}} \in \cO^{1 \times r}_{U \cap U_{\alpha}} 
 \cdot L_{\alpha}|_{U \cap U_{\alpha}} 
 \text{ for all } \alpha \in A,
$$
i.e., if and only if {\em the function  $h_{\alpha}(u) : = f(u)
L_{\alpha}(u)^{-1}$ 
is holomorphic on $U \cap U_{\alpha}$ for all $\alpha$.} 
It then follows that $h_{\alpha} L_{\alpha} = h_{\beta} L_{\beta}$ 
($= f$) on $U \cap U_{\alpha} \cap U_{\beta}$, and hence
\begin{equation}  \label{correspondence}
  h_{\beta} = h_{\alpha} L_{\alpha} L_{\beta}^{-1} \text{ on } U \cap 
  U_{\alpha} \cap U_{\beta}.
\end{equation}
This observation has the implication that the functions 
$\{ h_{\alpha}^{\top}\}_{\alpha \in A}$ piece together to form a
holomorphic 
section for the bundle $E_\cD$  (with sections locally identified
with column vector functions) dual to $E_{\cD}^{*}$  and defined via
the collection of transition functions
$$
   \Phi_{\alpha, \beta}^{*} = (\Phi_{\alpha, \beta}^{\top})^{-1}=
(L_{\alpha}^{\top})^{-1} ((L_{\beta}^{\top})^{-1})^{-1} \in 
   \cO^{r \times r}_{U_{\alpha} \cap U_{\beta}}.
$$
From the correspondence $f|_{U_{\alpha}} \mapsto h_{\alpha}^{\top} = (
L_{\alpha}^{\top})^{-1}\left( f|_{U_{\alpha}}\right)^{\top}$ derived 
above, we see that 
then the matrix-divisor space $\cL_*(\cD)$ is in one-to-one 
correspondence with holomorphic sections of the bundle $E_{\cD}$
(with local sections given in the more 
conventional form of column-vector functions).

So far, for a given right matrix divisor $\cD$, we have obtained an 
equivalence between the matrix-divisor space $\cL_{*}(\cD)$ and the
vector 
bundle $E_{\cD}^{*}$.  We now explain how one can associate a factor
of 
automorphy $\zeta_{\cD}$ on the universal cover $\widehat M$ of $M$ 
with any divisor $\cD$.  We let $\widehat M$ be the universal cover 
of $M$ with projection map $\rho \colon \widehat M \to M$. Let
$\{ U_{\alpha} \}$ be the cover described above determining the 
transition functions $\Phi_{\alpha,\beta}= L_{\alpha}L_{\beta}^{-1}$ 
defining the bundle $E_{\mathcal{D}}.$ It can be assumed that this
cover 
is chosen so that for all $\alpha$, the set $\widehat{U}_{\alpha} =
\rho^{-1}(U_{\alpha})$ is the 
disjoint union $\cup\{T\widehat{V}_{\alpha} : T\in \mathcal{G}\},$
where $\widehat{V}_{\alpha}$ is 
a fixed component of $\widehat{U}_{\alpha}.$ For $s\in
\widehat{U}_{\alpha}$ define $\widehat{L}_{\alpha}(s)=L_{\alpha}(\rho
(s)).$  Then the collection of transition functions
$\widehat{\Phi}_{\alpha,\beta} =
\widehat{L}_{\alpha}\widehat{L}_{\beta}^{-1}$ defines a vector bundle
$\widehat{E}_{\mathcal{D}}$ on $\widehat{M}$.   By a theorem of
Grauert
\cite{Grauert}, the bundle  $\widehat{E}_{\mathcal{D}}$  is
holomorphically equivalent to the trivial bundle. Thus for each
$\alpha$ there exists an invertible
 holomorphic matrix functions $H_{\alpha}$ on $U_{\alpha}$ such that
\[H_{\alpha}\widehat{L}_{\alpha}
 \widehat{L}_{\beta}^{-1}H_{\beta}^{-1} = I.\] The meromorphic matrix
function defined on $\widehat{M}$ by
 $F_{\mathcal{D}}(u) =H_{\alpha}(u) \widehat{L}_{\alpha}(u)$ for
$u\in \widehat{U}_{\alpha}$ is a 
 trivialization of the left factor of automorphy
\begin{equation}\label{factor}
    \zeta_{\mathcal{D}}(T,u)
= F_{\mathcal{D}}(Tu)F_{\mathcal{D}}^{-1}(u)\end{equation} 
 associated with the divisor $\mathcal{D}$, i.e., one can verify: 
 {\em $F_{\cD}$ is a well-defined invertible $r \times 
 r$ matrix-valued meromorphic function on $\widehat M$ satisfying
 (1) $F_{\cD}(Tu) = \zeta_{\cD}(T, u) F_{\cD}(u)$
 for all $u \in \widehat M$, and (2) $F_{\mathcal{D}}$ 
 has right null-pole divisor  $\mathcal{D}$ at points in
$\rho^{-1}(q_{0})$ for all $q_{0} \in M_{0}$} (with compatible choice 
of local coordinates as in Remark \ref{R:liftdata}), i.e., we have obtained 
an (albeit not particularly constructive) solution of the Third 
Interpolation Problem mentioned in the Introduction.

Thus the matrix-divisor space $\cL_{*}(\cD)$ lifts to the space $\widehat
\cL_{*}(\cD)$ 
consisting of all meromorphic row-vector functions $\widehat f$ on 
$\widehat M$ such that (1) $\widehat f(q) \in 
\cS(\Upsilon_{q_{0}},q_{0})$ for all $q_{0} \in \widehat M$, and (2)
$\widehat f$ is relatively automorphic with factor of automorphy 
$\zeta_{\cD}$:
$$
  \widehat f(Tu) = \zeta_{\cD}(T,u) \widehat f(u)  \text{ for 
  all } T \in \cG.
$$
Now that we have identified a global solution $F_{\cD}$ of the 
interpolation problem for divisor $\cD$ (lifted to $\widehat M$), we 
may adjust the correspondence \eqref{correspondence} between
$\cL_{*}(\cD)$ and the space of holomorphic section of $E_{\cD}^{*}$ 
to the global 
form
\begin{equation} \label{globalcorrespondence}\widehat f \mapsto
\widehat  h: =(F_{\cD}^{\top})^{-1}\widehat f^{\top}\end{equation}
which puts the elements of $\widehat L_*(\cD)$ in one-to-one 
correspondence with global holomorphic column-vector functions
$\widehat h$ 
which are relatively automorphic with (left) factor of automorphy 
$\zeta_{\cD}^{*} : = (\zeta_{\cD}^{\top})^{-1}$:  $\widehat h(Tu) =
\zeta_\cD^*(T,u) \widehat h(u)$.

In short, the bundle $E_{\cD}^*$  corresponds to the factor of
automorphy $\zeta_{\cD}$ on $\widehat{M},$ given 
as in \eqref{factor} in terms of the trivialization $F_{\cD},$  with
the space of holomorphic sections 
$H^{0}(\zeta_{\cD}^{*})$ of the dual factor of automorphy
$\zeta_{\cD}^{*}$ being holomorphic vector functions 
as in \eqref{globalcorrespondence}.

\subsection{Trivializations of  flat factors of automorphy via
divisor 
constructions}  \label{S:4.2}
In the case where $\deg \mathcal{D} =0$, a method developed in
\cite{BCV} 
leads to a condition sufficient for the existence of a
trivializations of 
$\zeta_{\mathcal{D}}$ where the entries have only limited poles over 
points on $M$ in a nonspecial divisor of degree $g$. In order to
formulate this result, we first recall some
basic results about holomorphic vector bundles.

A nonnegative degree $g$ divisor $D=p_{1} + \cdots + p_{g}$ on the closed
Riemann surface of genus $g$ is called a {\em nonspecial divisor} in case
$i(D)=0$, where $i(D)$ is the dimension of the 
space of meromorphic 1-forms $\omega$ whose divisor $(\omega)$
satisfies $(\omega)\geq D$. As is customary, 
the collection of nonnegative divisors of degree $g$ can be
identified with the (topological) $g$-fold 
symmetric product $M^{(g)}$. The collection of nonspecial divisors
forms an open subset of  $M^{(g)}$ and 
given the nonspecial divisor $D$,  there is a nonspecial divisor
$D'=p_{1}+\cdots+p_{g}$ close to $D$ where 
the points $p'_{1},\ldots, p'_{g}$ are distinct \cite[page 91]{FarkasKra}.
It follows from the Riemann-Roch Theorem
\cite[page 73]{FarkasKra}  that the nonnegative degree $g$ divisor
is nonspecial if and only if the degree $g-1$ divisor 
$D_{0} = D - p_{0}$ satisfies $h^{0}(\lambda_{D_{0}})=0,$ where as is
customary, $h^{0}(\lambda_{D_{0}})$ is 
used to denote the dimension of the space of holomorphic sections of
the line bundle $\lambda_{D_{0}}$ 
associated with the divisor $D_{0}$ or, equivalently, the dimension
$l(D_{0})$ of the linear space $L(D_{0})$  
of meromorphic functions $f$ on $M$ whose divisors satisfy
$(f)+D_{0}\geq 0$;  here and in the sequel, $p_{0}$ is any point 
disjoint from $p_{1}, \dots, p_{g}$. In case $D$ is nonspecial, 
then by the Riemann-Roch Theorem the degree $g-1$ divisor
$D_{0} = D - p_{0}$  
satisfies \[ h^{0}(\lambda_{D_{0}}) =
h^{0}(\lambda_{D_{0}}^{-1}\kappa) = 0,\] where $\kappa$ is the
canonical line bundle. Applying the Jacobi inversion theorem (see 
e.g.~\cite[page 97]{FarkasKra}) to the divisor 
$(\lambda_{D_{0}}^{-1}\kappa)+p_{0}$,  one obtains a linearly
equivalent degree $g$
divisor $\widetilde D =  \widetilde{p}_{1}+\cdots+
 \widetilde{p}_{g}$.  
The degree $g-1$ divisor
 \[
 \widetilde{D}_{0} =\widetilde{p}_{1}+\cdots+ \widetilde{p}_{g} - p_{0}
 \] 
 is equivalent to the
divisor $(\lambda_{D_{0}}^{-1}\kappa)$ and  therefore also satisfies
$l(\widetilde{D}_{0}) = 0$. As noted above, it follows that 
$\widetilde D =  \widetilde{p}_{1}+\cdots+ \widetilde{p}_{g}$ is nonspecial.

It is a consequence of Weil's characterization of flat bundles
\cite{Weil} (see also \cite[page 110]{Gunvect}) 
that a rank $r$ bundle $E$ of degree zero
over $M$ will be flat in case, for some line bundle $\lambda $ of
degree $g-1$, one has $h^{0}(\lambda\otimes E) = 0$.
In case of genus one, this condition is also sufficient.
In particular, if the line bundle $\lambda$ has the form 
$\lambda_{D_{0}}$ where $D_{0} = p_{1} +  \cdots + p_{g} - p_{0}$ with $p_{1} + \cdots + 
p_{g}$ nonspecial and $E$ is a rank-$r$ vector bundle of degree zero 
over $M$ such that
\begin{equation}   \label{sect}
   h^{0}(\lambda_{D_{0}}\otimes E) = 0,
    \end{equation}
it follows that $E$ is flat.  In such a case we will say that the 
degree-0 vector bundle $E$ has the property NSF (nonspecial flat).
To summarize the preceding discussion in terms of the notion of
NSF bundles, we see that any NSF bundle is flat and 
in the genus 1 case the classes of flat and NSF bundles coincide. The 
goal of this section (see Theorem \ref{T:NSF} below) is 
to show that, given a flat factor of automorphy $\zeta$ associated 
with an NSF bundle $E_{\zeta}$, there is a
meromorphic matrix function $F$ automorphic with respect 
to $\zeta$ having matrix divisor supported on only $g+1$ points.
Before arriving at this result, we need to go through a 
number of preliminaries.

The property NSF  is symmetric with respect to bundle adjoints on
the collection of degree zero bundles. Indeed, 
if $E$ has property NSF and 
\eqref{sect} holds, then it follows from the Riemann-Roch Theorem for
vector bundles (see  \cite[page 64]{Gunvect})  
that   
 \[
  h^{0}((\lambda_{D_{0}})^{-1}\kappa\otimes  E^{*}) = 0.
  \] 
  As mentioned earlier, the line bundle $\lambda_{D_{0}}^{-1}\kappa$
is equivalent to a degree $g-1$ line bundle $\tilde{\lambda}$ 
  associated with a divisor of the form 
  \[
  \widetilde{D}_{0} =  \widetilde{p_{1}}+\cdots +
  \widetilde{p_{g}} -p_{0},
\] 
where the nonnegative degree $g$ 
  divisor $ \widetilde{p_{1}}+\cdots + 
\widetilde{p_{g}}$ is nonspecial. Thus
$E^{*}$ has property NSF. 

\begin{remark} {\rm As noted in \cite{BCV}, there are examples in
higher genus of degree $0$ (even semi-stable) rank $2$ 
    bundles where $h^{0}(\lambda\otimes E) \neq 0$ for all degree
$g-1$ line bundles $\lambda$.}
\end{remark}

Following \cite{BCV}, we introduce automorphic meromorphic 
matrix functions on
$\widehat{M}$ that have a prescribed (left) pole 
pair $(A_{\pi}, C_{\pi})$ at the point $\mathcal{G}w$ where $w$ is a
fixed point in the fundamental domain $M_{0}$. To 
this end let $D_{0}$ 
be a divisor of degree $g-1$ on $M$ of the form
$D_{0} = p_{1}+\cdots + p_{g} -p_{0}$, where 
$p_{1}+\cdots + p_{g}$ is nonspecial and such that the points
$p_{0},p_{1},\ldots,p_{g}$ are distinct from $w$.  We also fix a 
local coordinate $z$ for $M_{0}$ centered at $w$ (so $z(w) = 0$).  
Since $h^{0}(\lambda_{D_{0}}) = 0,$ it follows that for any integer
$k\geq 1$ there is a unique meromorphic 
function $f^{D_{0}}_{kw}$ on $M$, equivalently, an automorphic
meromorphic function on $\widehat{M}$, whose divisor satisfies 
$(f^{D_{0}}_{kw})+D_{0}+kw\geq 0$ and which is normalized so that the
principal part of the Laurent series at $w$ with respect to the local 
coordinate $z$ at $w$ has 
the form $z(u)^{-k}$ (see \cite[page 147]{BCV}).
Suppose that $A$ is the $n\times n$ Jordan cell 
$$
 A = \begin{bmatrix} 0 & 1 & & \\ & \ddots & \ddots & \\ & & 0 & 1 \\ & 
 & & 0 \end{bmatrix}.
 $$
Introduce the matrix function 
\begin{equation}  \label{fD0wA} 
f^{D_{0}}_{w,A}=\left[ 
\begin{array}{lllll}
f^{D_{0}}_{w}  & f^{D_{0}}_{2w} & f^{D_{0}}_{3w} & \cdots  &
f^{D_{0}}_{nw} \\ 
 & f^{D_{0}}_{w}  & f^{D_{0}}_{2w}& \ddots  & \vdots  \\ 
 &  & \ddots  & \ddots  & f^{D_{0}}_{3w} \\ 
 &  &  &  \ddots  & f^{D_{0}}_{2w}\\ 
 &  &  &  & f^{D_{0}}_{w} 
\end{array}
\right] 
\end{equation}
where as usual unspecified entries are equal to $0$.
This definition is extended to an arbitrary nilpotent matrix $A$ by setting
$f^{D_{0}}_{w,SAS^{-1}}=Sf^{D_{0}}_{w,A}S^{-1}$ and 
$f^{D_{0}}_{w,A_{1}\oplus A_{2} } = f^{D_{0}}_{w,A_{1}}\oplus
f^{D_{0}}_{w,A_{2}},$ where $S$ is an invertible matrix. 
It follows that for any nilpotent matrix $N$ the difference
\[
f^{D_{0}}_{w,N}(u)-((z(u)I-N)^{-1}
\] 
is analytic at $u=w$. 
This last statement follows from the fact that, when $N$ is $r\times
r$ Jordan cell with eigenvalue $0$,
the local (left) pole pair of $f^{D_{0}}_{w,N}$ 
at $w$ has the form $(A_{\pi}, C_{\pi})=(N,I_{r})$. The only other
poles of entries of the meromorphic matrix function 
$f^{D_{0}}_{w,N}$ are at the points $p_{1}, \dots, p_{s}$ of the
divisor $D_{0}$. In fact, if the divisor $D_{0}$ is 
written in the form 
\[
D_{0}=n_{1}p_{1}+\cdots+n_{s}p_{s} - p_{0},
\]
where the distinct points $p_{1},\ldots,p_{s}$ 
appear with the positive multiplicities $n_{1},\ldots,n_{s}$,
respectively, then at  $p_{1},\ldots,p_{s}$ the entries 
of $f^{D_{0}}_{w,N}$  have poles of order at most
$n_{1},\ldots,n_{s}$, respectively.

Let the degree $g-1$ divisor $D_{0}=p_{1}+\cdots+p_{g}-p_{0}$ be as
above with points $p_{0},p_{1},\ldots,p_{g}$ 
distinct from points $u_{1}, \dots, u_{k}$ on which $\cD$ is supported.
For $\cD$ a right matrix null-pole divisor supported at the points 
$u_{1}, \dots, u_{k}$ as in \eqref{data} with the set of indices 
$\{1, \dots, k\}$ partitioned into subsets I, II, III as in 
\eqref{data''}, we introduce the space 
$\mathcal{M}_{\mathcal{D}}^{D_{0}}$ of $r$-dimensional row-vector
meromorphic functions on $M$ given by
\begin{equation} \label{row} 
    \mathbf{k} =\mathbf{u}_{0} +\sum_{ i \in {\rm I} \cup 
    {\rm II}}
    \mathbf{u}_{i}f^{D_{0}}_{u_{i},A_{\pi_{u_{i}}}} C_{\pi_{u_{i}}},
\end{equation} 
where 
    $\mathbf{u}_{0} \in {\mathbb C}^{1 \times r}$ and
$\mathbf{u}_{i} \in {\mathbb C}^{1 \times n_{\pi_{u_{i}}}}$, 
$i=1,\ldots, n_{P}$.  The linear map $\mathbf{T}$ 
    from  ${\mathbb C}^{1 \times (r+n_{u_{1}} + \cdots +
n_{u_{n_{P}}})}$
to  $\cM_{\cD}^{\cD_{0}}$
that associates the 
 vector $\begin{bmatrix} \mathbf{u}_{0} & \mathbf{u}_{1} &  \cdots 
& \mathbf{u}_{n_{P}} \end{bmatrix}$ with the vector $\mathbf{k}$
given by (\ref{row}) is one-to-one.
 Indeed, if 
 \[
 \mathbf{k} = \mathbf{T}( \begin{bmatrix} \mathbf{u}_{0} &
\mathbf{u}_{1} & \cdots &
   \mathbf{u}_{n_{P}} \end{bmatrix}) = \mathbf{0}, 
   \] 
then the singular part of $\mathbf{k}$ at $u_{i}$ or, 
    equivalently, the singular part of
$$
\mathbf{u}_{i}( z_{i}(u)I-A_{\pi{i}})^{-1}C_{\pi_{u_{i}}}
$$
(where $z_{i}$ is the local coordinate at $u_{i}$) at $u_{i}$, is zero. 
    This leads one to conclude that
$\mathbf{u}_{i}A_{\pi_{u_{i}}}^{j}C_{\pi_{u_{i}}} = \mathbf{0}$ for
$j=0,1,\ldots, n_{\pi_{u_{i}}} - 1$.  The controllability of the pair 
$(A_{\pi_{u_{i}}}, C_{\pi{u_{i}}})$ for each $i=1,\ldots,n_{P}$  
implies  that $\mathbf{u}_{i} =  \mathbf{0}$.  As $\mathbf{u}_{0}$ is
the value of $\mathbf{k}$ at $p_{0}$  
(since each $f_{u_{i},A}^{D_{0}}$ vanishes at $p_{0}$), we get
$\mathbf{u}_{0}=0$ as well.  We conclude that 
$\mathbf{T}$ is one-to-one.

When one assumes that the meromorphic matrix function $\bk$ has the 
form \eqref{row}, there is a convenient test for identifying when 
$\bk \in \cO_{u_{i}}^{1 \times r} \cdot F_{\cD}$ for each of the points 
$u_{i}$ ($i = 1, \dots,k= n_{\infty} + n_{c} + n_{0}$) in the support 
of the matrix divisor $\cD$. Before presenting this result we need 
some additional notation.

For $i \in {\rm I} \cup {\rm II}$ and $j \in {\rm II} \cup {\rm 
 III}$  and $i \ne j$ in case both $i$ and $j$ are in ${\rm II}$, let 
 \begin{equation}   \label{gamma-a}
\Gamma^{D_{0}}_{\cD,ij} =  -\text{res}_{u=u_{j}}[ 
 f^{D_{0}}_{u_{i},A_{\pi_{u_{i}}}}(u)
 C_{\pi_{u_{i}}}B_{\zeta_{u_{j}}}((z_{j}(u)I-A_{\zeta_{u_{j}}})^{-1}] 
 \end{equation}
and for $i$ and $j$ both in ${\rm II}$ with $i = j$, let 
 \begin{align}
&     \Gamma^{D_{0}}_{\cD,ij} = \Gamma^{D_{0}}_{\cD,jj}   
 =:    S_{u_{j}} \notag \\
&   -  \text{res}_{u_{j}} 
     \left[\left(f^{D_{0}}_{u_{j},A_{\pi_{u_{j}}}}(u) 
     - (z_{j}(u)I - A_{\pi_{u_{j}}})^{-1} \right) C_{\pi_{u_{j}}} 
  B_{\zeta_{u_{j}}}(z_{j}(u)I-A_{\zeta_{u_{j}}})^{-1} \right].
  \label{gamma-b}
 \end{align}
These matrices form the block entries for an $n_{P}\times n_{Z}$
block matrix 
$\Gamma^{D_{0}}_{\cD}$ having rows indexed by $i \in {\rm I} \cup
{\rm II}$ 
and columns indexed by $j \in {\rm II} \cup {\rm III}$:
\begin{equation}   \label{gamma}
   \Gamma^{D_{0}}_{\cD} = [ \Gamma^{D_{0}}_{\cD, ij}]_{i \in {\rm I}
\cup {\rm 
   II}, j \in {\rm II} \cup {\rm III}}.
\end{equation}
We are now able to describe and prove the following result.

\begin{prop}   \label{P:test}  Assume that the meromorphic row-vector 
    function $\bk$ has the form \eqref{row}.  Then $\bk \in 
    \cO_{u_{i}}^{1 \times r} F_{\cD}$ for each point $u_{i}$ in the 
    support of $\cD$ if and only if 
\begin{equation}   \label{test}
    \bu_{0} B_{\zeta} = {\rm row}_{i \in {\rm I} \cup {\rm II}}
[\bu_{i}]
    \cdot \Gamma^{D_{0}}_{\cD}
\end{equation}    
where $B_{\zeta} =     {\rm row}_{i \in {\rm II} \cup {\rm III}} 
[B_{\zeta_{i}}]$.
\end{prop}

\begin{proof}
    Suppose $\mathbf{k}$ is given in the form (\ref{row}) 
and one wishes to investigate if this meromorphic row-vector function
belongs to $\mathcal{O}^{1\times r}_{u}\cdot F_{\mathcal{D}}$ at
$u=u_{j}$ for $j \in {\rm I} \cup {\rm II} \cup {\rm III}$. 

If $j \in {\rm I}$, no extra 
condition is required:  ${\mathbf k}$ is already of the correct form 
to be in the $\cO^{1 \times r}_{u_{j}} \cdot F_{\cD}$. 

For $j \in {\rm III}$, then using the
description (\ref{local}) one sees that the germ  
 of $\mathbf{k}$ belongs to $\mathcal{O}^{1\times r}(\{u_{j}\}) \cdot
F_{\mathcal{D}}$ if and only if 
 \[
 \text{res}_{u=u_{j}}\left\{ \left[\bu_0+\sum_{i=1}^{n_{P}}\mathbf{u}_{i}  
f_{w_{i},A_{\pi_i}}^{D_{0}}(u)C_{\pi_{i}} \right]
B_{\zeta_{j}}(z_{j}(u)I-A_{\zeta_{j}})^{-1} \right\}
=0,
\] 
or, equivalently,  
 \begin{equation} \label{interp} 
    \bu_{0} B_{\zeta_{j}}  =\sum_{i=1}^{n_{P}}\mathbf{u}_{i}  
    \Gamma_{\cD, ij}^{D_{0}}.
\end{equation} 

For $j \in {\rm II}$, we rewrite ${\mathbf k}$ near $u_{j}$ as
\begin{align*}
 {\mathbf k} = & [ {\mathbf u}_{j} (z_{j}(u)I - A_{\pi_{u_{j}}})^{-1} 
 C_{\pi_{j}}] \\
 & + \left[ \left( \bu_{0} + \sum_{i \in {\rm I} \cup {\rm II}}
{\mathbf u}_{i} 
 f^{D_{0}}_{u_{i}, A_{\pi_{u_{i}}}} C_{\pi_{u_{i}}}\right)
 -  {\mathbf u}_{j} (z_{j}(u) I - A_{\pi_{u_{j}}})^{-1} C_{\pi_{j}} 
 \right].
\end{align*}
Note that the second bracketed term in this decomposition is 
holomorphic at $u_{j}$ since the singular parts at $u_{j}$ cancel out.
We then apply the characterization \eqref{local} to this 
decomposition to see that ${\mathbf k} \in \cO_{u_{j}}^{1 \times k} 
F_{\cD}$ if and only if
\begin{align*}
\text{res}_{u=u_{j}} & \left\{ \left[ \left(  \bu_{0} + \sum_{i \in
{\rm I} \cup {\rm II}} 
{\mathbf u}_{i} f^{D_{0}}_{u_{i}, A_{\pi_{u_{i}}}}
C_{\pi_{u_{i}}}\right)
 -  {\mathbf u}_{j} (z_{j}(u) I - A_{\pi_{u_{j}}})^{-1} C_{\pi_{j}} 
 \right] \right. \cdot   \\
& \cdot  B_{\zeta_{u_{j}}} (z_{j}(u)I - A_{\zeta_{u_{j}}})^{-1} 
 \big\} = {\mathbf u}_{j} S_{u_{j}}.
\end{align*}
This condition collapses to \eqref{interp} for the case where $j \in 
{\rm II}$. 

When we arrange the conditions \eqref{interp} as an equality of two 
block row matrices (with block-row entries indexed by $j \in {\rm II} 
\cup {\rm III}$), we arrive at the single matrix equation 
\eqref{test}.  As the analysis is necessary and sufficient, the 
result of Proposition \ref{P:test} follows.
 \end{proof}   

The subspace  $(\mathcal{M}_{\mathcal{D}}^{D_{0}})_{0}$ of
$\mathcal{M}_{\mathcal{D}}^{D_{0}}$ consisting of the 
meromorphic functions $\mathbf{k}$ on $M$ of the form (\ref{row})
with $\mathbf{u}_{0} =
\mathbf{0}$ was used in \cite{BCV} to 
describe the holomorphic sections of $(\lambda_{-D_{0}}\otimes
E_{\mathcal{D}})^{*}\cong\lambda_{D_{0}}\otimes
E_{\mathcal{D}}^{*}$. In fact, what is described directly is the 
matrix-divisor subspace  $\cL_{*}((-D_{0}) 
\otimes \cD)$ for the right matrix divisor $(-D_{0}) \otimes \cD$.  
By the discussion in Section \ref{S:4.1}, we 
see that the space $\cL_{*}((-D_{0}) \otimes \cD)$ is isomorphic to
the 
space of holomorphic sections for the bundle $(\lambda_{-D_{0}}
\otimes 
E_{\cD})^{*}.$ That is, after taking transposes, the space of
sections of the bundle $\lambda_{D_{0}} \otimes E_{\cD}^{*}$  is seen
to be equivalent to the space $\cL_{*}((-D_{0})\otimes \cD)$ of
multivalued functions 
${\mathbf k}$ on $M$ satisfying ${\mathbf k}  \in \cO_{u}^{1 
\times r} \cdot F_{\cD} f_{D_{0}}^{-1}$ where $f_{D_{0}}$ is a 
trivialization of the line bundle $\lambda_{D_{0}}$ (with classical 
line bundle conventions).  As noted in \cite[page 149]{BCV}, 
$\cL_{*}((-D_{0}) \otimes \cD)$ can be identified concretely as a
subspace of
$(\mathcal{M}_{\mathcal{D}}^{D_{0}})_{0};$ the result is as follows.

\begin{prop}\label{P:nonspecialint}
     Let $D_{0}$ be a degree $g-1$ divisor of the form
$n_{1}p_{1}+\cdots + n_{s}p_{s} - p_{0},$ where $D=n_{1}p_{1}+ \cdots 
+ n_{s}p_{s}$      
is a nonspecial divisor and let the degree zero null-pole
divisor $\mathcal{D}$ be of the form (\ref{data}) with partitioning 
of indices as in \eqref{data''}.  Define the block matrix 
$\Gamma^{D_{0}}_{\cD}$ as in \eqref{gamma}. Then the matrix-divisor space 
$\cL_{*}((-D_{0}) \otimes \cD)$ of the divisor $(-D_{0}) \otimes
\cD$, 
an isomorphic copy of the space of holomorphic sections of the bundle 
$\lambda_{D_{0}}\otimes E_{\cD}^{*}$, is equal to the
subspace of row vector meromorphic functions  $\mathbf{k}$ in  
$(\mathcal{M}_{\mathcal{D}}^{D_{0}})_{0}$, i.e., $\bk$ has the form 
\eqref{row} with $\bu_{0} = 0$
\begin{equation}   \label{row0}
    \mathbf{k} =  \sum_{ i \in {\rm I} \cup {\rm II}}
    \mathbf{u}_{i}f^{D_{0}}_{u_{i},A_{\pi_{u_{i}}}} C_{\pi_{u_{i}}},
\end{equation} 
 where the row vector $\mathbf{u} = {\rm row}_{i \in {\rm I} \cup 
 {\rm II}}[ \bu_{i}]$ satisfies 
\begin{equation}   \label{test'}
    \mathbf{u}\Gamma^{D_{0}}_{\cD}=\mathbf{0}.
    \end{equation}
In particular,
$h^{0}(\lambda_{D_{0}}\otimes E_{\mathcal{D}}^{*})$ equals 
     the dimension of the left-kernel of $\Gamma^{D_{0}}_{\cD}$
acting on 
$\mathbb{C}^{1 \times (n_{\pi_{u_{1}}} + \cdots +  
n_{\pi_{u_{n_{P}}}})}$.
\end{prop}

\begin{proof}
    By assumption the points $p_{0}, p_{1}, \dots, p_{g}$ are assumed 
    to be disjoint from the points $u_{i}$ ($i \in {\rm I} \cup {\rm 
    II} \cup {\rm III}$) of the support of $\cD$.  Thus the criterion 
    \eqref{test} from Proposition \ref{P:test} (applied with $\bu_{0} 
    = 0$) informs us that any $\bk$ of the form \eqref{row0} is in
    $\cO^{1 \times r}_{u_{i}}\cdot F_{\cD} f_{D_{0}}^{-1}$ for each point 
    $u_{i}$ in the support of $\cD$.
Moreover, since each function $f^{D_{0}}_{u_{i}, A_{\pi_{u_{i}}}}$
($i \in {\rm 
I} \cup {\rm II}$) has a simple pole at each $p_{i}$ and a simple 
zero at $p_{0}$, we see that any ${\mathbf k}$ of the form 
\eqref{row0} is in $\cO^{1 \times r}_{p_{j}} \cdot F_{\cD} f_{D_{0}^{-1}}$ 
for $j=0,1, \dots, g$ as well. 

Conversely we argue that any function ${\mathbf k}'$ in $\cL_{*}(
(-D_{0}) 
\otimes \cD)$  necessarily has the form \eqref{row0} as follows.  We 
choose vectors $\bu_{1}, \dots, \bu_{n_{P}}$ so that the function 
${\mathbf k}$ given by \eqref{row0} has the property that its 
singularities at the points $u_{1}, \dots, u_{n_{P}}$ exactly cancel 
the singularities of ${\mathbf k}'$ at these points, i.e., so that
$$
  {\mathbf k} - {\mathbf k}' \text{ is analytic on } M \setminus \{ 
  p_{0}, p_{1}, \dots, p_{g}\}.
$$
From the assumptions on ${\mathbf k}$ and ${\mathbf k}'$, 
it follows that each matrix entry of ${\mathbf k} - {\mathbf k}'$ is 
in the classical divisor space $\cL(D_{0})$.  As $D$ is nonspecial, 
it follows that ${\mathbf k} - {\mathbf k}' = 0$, i.e., ${\mathbf k} 
= {\mathbf k}'$ is in the space $(\cM^{D_{)}}_{\cD})_{0}$.
\end{proof}

The content of Proposition \ref{P:nonspecialint} is essentially the 
same as  Theorem 6 from \cite{BCV}; we note that our analysis here
also corrects some misprints in the formula 
for $\Gamma^{\lambda}$ given in \cite[page 149]{BCV}.

If the divisor $\cD$ given as in (\ref{data}) has degree zero, then
it follows from Proposition \ref{P:nonspecialint}
that a necessary and sufficient
condition for the bundle $E_{\mathcal{D}}$  
(or, equivalently, $E_{\cD}^{*}$) to have property NSF is that
there be a choice of nonspecial divisor
$D=p_1+\cdots + p_g$ so that, with $D_0 = p_1+\cdots + p_g -p_0,$ the 
associated matrix $\Gamma^{D_{0}}_{\cD}$ is invertible.

The matrix $\Gamma^{D_{0}}_{\cD}$  can be used to give criteria for
the
existence of a {\em nondegenerate} (i.e., with determinant not
vanishing 
identically) $r\times r$-meromorphic matrix function 
with controlled pole behavior. To see this let $\mathcal{D}$ be given
as in (\ref{data}) and $D_{0} =p_{1}+\cdots +p_{g}-p_{0}$ 
be a divisor of degree $g-1$ with $D=p_{1}+\cdots +p_{g}$ a
nonspecial divisor. As usual we assume that the points
$p_{0},p_{1}\ldots, p_{g}$ of $D_{0}$ are distinct from the points
$u_{1},\ldots, u_k$ ($k=n_{\infty} + n_{c} + n_{0}$) of the divisor
$\mathcal{D}.$ 
Introduce the $r\times r$ meromorphic matrix function
\begin{equation} \label{matrix} 
 K = U_{0} + \sum_{ i \in {\rm I} \cup {\rm II}} U_{u_{i}}
f^{D_{0}}_{u_{i},A_{\pi_{u_{i}}}}C_{\pi_{u_{i}}},
\end{equation} 
where the matrix $U_{i} \in {\mathbb C}^{r\times n_{\pi_{{u_{i}}}}}$ 
for $i=1,\ldots, n_{P}$ and $U_{0}$ is $r\times r$.

\begin{prop} \label{P:Kexists} 
    Let $\mathcal{D}$ be a rank $r$ degree zero right matrix null-pole
divisor on $M$ of the form (\ref{data}) and   let $D_{0}$ be a 
divisor of degree $g-1$ of the form
    $D_{0}= n_{1} p_{1}+\cdots +  n_{s} p_{s}-p_{0}$ 
    ($p_{0}, p_{1}, \dots, p_{s}$ taken to be distinct with  
    multiplicities $n_{j} \ge 0$, $j=1, \dots, s$)  satisfying $n_{1} 
    + \cdots + n_{s} = g$) such that 
$n_{1}p_{1}+\cdots + n_{s}p_{s}$ is a nonspecial divisor 
    and  the points $p_{0},p_{1},\ldots, p_{s}$ are all distinct from
the points $u_{1},\ldots, u_k (k= n_{\infty} + n_{c} + n_{0})$.  Let
$K$ be the
    meromorphic matrix function given in (\ref{matrix}).  The
condition 
\begin{equation} \label{exists}  
 U_{0} \cdot {\rm row}_{j \in {\rm II} \cup {\rm III}} [
B_{\zeta_{n_{\zeta_{j}}}}] 
 = {\rm row}_{i \in {\rm I} \cup {\rm 
 II}}[ U_{i}]  \cdot  \Gamma_{\cD}^{D_{0}} 
\end{equation}  
is necessary and sufficient for the germ  of the meromorphic matrix
function $K$
given by (\ref{matrix}) at $u_{0}$ to satisfy
 \begin{equation} \label{germK}
     \cO_{u_0}^{r \times r} K  \subset  \mathcal{O}_{u_{0}}^{r\times
r} F_{\mathcal{D}} 
     \text{ for all } u_{0} \in M \setminus \{p_{1}, \dots, p_{g}\}.
\end{equation}
In particular, if $\Gamma^{D_{0}}_{\cD}$ 
    is invertible, then for a fixed $r\times r$-matrix $U_{0}$ there
exists a unique meromorphic matrix function $K$ satisfying
\begin{enumerate}
    \item $K(p_{0}) = U_{0}$,
    \item $K$ satisfies \eqref{germK}, and
    \item each entry of $K$ has a possible pole at $p_{j}$ of order 
    at most $n_{j}$ for $j = 1, \dots, s$.
\end{enumerate}
\end{prop}

\begin{proof} 
The existence part of the result follows from the condition
(\ref{interp}) which gives
necessary and sufficient conditions for the rows of $K$ to 
belong to $\mathcal{O}_{u_{0}}^{1\times r}\cdot F_{\mathcal{D}}.$  We
omit further detail.  As for the uniqueness, the conditions imply 
that each entry of $K$ has its only poles in $\{u_{i} \colon i \in 
{\rm I} \cup {\rm II}\} \cup \{p_{1}, \dots, p_{s}\}$ with the
various 
multiplicities controlled by $\cD$ and $D_{0}$.  This forces $K$ to 
have the form \eqref{matrix}.  Specifying the value of $K$ at $p_{0}$ 
determines the matrix $U_{0}$.  The fact that $K \in \cO^{r 
\times r}_{u_{j}} F_{\cD}$ for $j \in {\rm II} \cup {\rm III}$ then
leads to the system of 
equations \eqref{exists} which is just a matrix version of 
\eqref{test}.
The assumption that $\Gamma^{D_{0}}_{\cD}$ is invertible then leads 
to the remaining coefficients $U_{1}, \dots, U_{n_{P}}$ 
being uniquely determined.
\end{proof}

Proposition \ref{P:Kexists} as presented here  is a corrected version
of Proposition 8 from
 \cite[p. 152]{BCV}. (The assertion that $F^{-1}$ is 
analytic off $\{ z_{1},\ldots,z_{k},p_{0},\ldots,p_{g}\}$ should not
have been included in the statement of Proposition 8 in 
\cite[p. 152]{BCV}.)

\smallskip

We can now establish the following result:

\begin{thm}   \label{T:NSF}
Let $E$ be a rank $r$ vector bundle of degree zero over the closed
Riemann surface $M$ of genus $g$ and let $\zeta$ be 
the corresponding factor of automorphy on $\widehat{M}$. If the
vector bundle $E$ has property NSF, then there exists a
nonspecial
divisor $D_{ns}$, which we write out in the more detailed form
$D_{ns} = n_{1}p_{1} +\cdots + n_{s}p_{s}$ where $p_{1}, \dots, 
p_{s}$ are distinct with respective multiplicities $n_{1}, \dots, 
n_{s}$ adding up to $g$, and a trivialization $F$ of
$\zeta$ such that the only poles of entries of $F$ 
are at points in $\widehat{M}$ over the points $p_{j}$ in the support 
of $D_{ns}$, with pole order at $p_j$ at most equal to  the 
multiplicity $n_{j}$ of $p_{j}$ in $D_{ns}$ for $j=1,\ldots s$.  
\end{thm}

\begin{proof} 
 Assume that $E$ has property NSF.  Then we may choose a nonspecial divisor 
 $D$ so that $h^{0}(\lambda_{D} \otimes E) = 0$  Let $G$ be a trivialization
of  the factor of automorphy $\zeta^{*}$ 
 associated with $E^{*}$, so that $G_{*}=(G^{\top})^{-1}$ is a
trivialization of the factor of automorphy $\zeta$  
 associated to $E$.   Let $\mathcal{D}$ be the right matrix divisor of
$G$ restricted to $M_{0}$. 
Note that $E^{*}\cong E^*_{\mathcal{D}}$.  Since 
$h^{0}(\lambda_{D_{0}} \otimes E_{\cD}) = h^{0}(\lambda_{D_{0}} 
\otimes E) = 0$, it follows from Proposition \ref{P:nonspecialint} 
that the matrix $\Gamma^{D_{0}}_{ \cD}$ is invertible.  We
may then apply Proposition \ref{P:Kexists} to get a uniquely 
determined nondegenerate meromorphic matrix function $K$  
satisfying properties (1), (2), (3) with
$U_{0} = I_{r}$ (or any invertible matrix). By taking transposes, 
we may conclude from property (2) that the germ of $K^{\top}$ belongs to
 $G^{\top}\cdot\mathcal{O}^{r\times r}_{u_0} =
(G_{*})^{-1}\cdot\mathcal{O}^{r\times r}_{u_0}$ for $u_0 \neq
p_1,\ldots, p_g.$ 
 Thus $G_{*}K^{\top}$ is analytic off  $p_1,\ldots, p_g.$  Since $K$
 is a nondegenerate meromorphic function, 
$F=G_{*}K^{\top}$ of $\zeta$ is also a trivialization of $\zeta$. 
The only poles of entries of $F$ are at the points $p_{j}$,  with
multiplicity at most $n_{j}$, $ j=1,\ldots,s$. 
Thus $F$ is a trivialization of $\zeta$  with the desired
properties. 
\end{proof}

The next result assures us  that  degree zero vector bundles 
having property NSF always have trivializations with pole behavior analogous to that of the
trivializations \eqref{tatiyah} when $g=1$.  The proof of this result 
depends on the fact that the entries of the matrix $\Gamma^{D_{0}}_{\cD}$ depend 
continuously on the divisor $D$; we postpone the proof of this latter result 
to Section \ref{S:explicit} where we discuss explicit formulas for 
the entries of the matrix $\Gamma^{D_{0}}_{\cD}$.

\begin{thm}  \label{T:NSF'}
Let $E$ be a rank $r$ vector bundle of degree zero over the closed
Riemann surface $M$ of genus $g$ and let $\zeta$ be 
the corresponding factor of automorphy on $\widehat{M}$. If the
vector bundle $E$ has property NSF, 
then there exists a non-special
divisor $p_{1} +\cdots +p_{g}$, where the points $p_1, \dots, p_g$
are distinct
and a trivialization $F$ of $\zeta$ such that the only poles of
entries of $F$ are simple poles at points in $\widehat{M}$ 
over the points $p_{1}, \dots, p_{g}$. 
\end{thm}

\begin{proof}  Suppose that the bundle $E$ has property NSF with 
    associated nonspecial divisor $D =  p_{1} + \cdots +  
    p_{g}$  and $D_{0} =  p_{1} + \cdots + p_{g} - p_{0}$ 
such that $h^{0}(\lambda_{D_{0}} \otimes E) = 0$.  Again by  Proposition 
    \ref{P:nonspecialint} we get that the matrix 
    $\Gamma^{D_{0}}_{\cD}$ is invertible.  
   A consequence of \cite[page 
    91]{FarkasKra} already noted is that the nonspecial divisor $D =
     g_{1} + \cdots + p_{g}$ can be approximated arbitrarily 
    well (in the topology of $M^{(g)}$) by a nonspecial divisor $D' = 
    p'_{1} + \cdots + p'_{g}$ with $p'_{1}, \dots, p'_{g}$ distinct.
    The block entries of the matrix 
    $\Gamma^{D_{0}}_{\cD}$ involve the building blocks 
    $f_{k w}^{D_{0}}$ as in formula \eqref{fD0wA}.  We keep the local 
    admissible Sylvester data sets $\Upsilon_{q_{0}}$ \eqref{localdata}
    fixed and perturb only the divisor $D_{0}$.  Then the matrix 
    entries of $\Gamma^{D_{0}}_{\cD}$ move continuously as long as 
    the canonical scalar functions $f^{D_{0}}_{k w}$ are continuous 
    with respect to the support $p_{1} + \cdots + p_{g} \in M^{(g)}$ 
    of $D_{0}$, which is precisely the content of Corollary 
    \ref{C:contdep} discussed below.
 Since invertibility is an open condition, it follows that 
$\Gamma^{D'_{0}}_{\cD}$ is again invertible as long as the point 
$( p'_{1}, \dots, p'_{g})$) is arranged to be sufficiently close to 
$(p_{1}, \dots, p_{g})$ in $M^{(g)}$.  We now 
use the construction in Theorem \ref{T:NSF} to see that $E$ has a 
trivialization $F$ having only possible poles occurring  at the points
$p'_{1}, \dots, p'_{g}$ with pole order at $p'_{j}$ at most $1$.
 \end{proof}

\subsection{Automorphic interpolants with given divisor} \label{S:4.3}
Let $\mathcal{D}$ be a degree zero null-pole divisor. We say that the 
meromorphic $r \times r$ matrix-valued function $F$ {\em interpolates 
the right matrix null-pole divisor} $\cD$ if
\begin{equation}  \label{eq}
    \cO^{r \times r}_{u_{0}} F = \cO^{r \times r}_{u_{0}} F_{\cD}
    \text{ for all } u_{0} \in M.
\end{equation}
An equivalent condition is that 
$$
\cO^{1 \times r}_{u_{0}} F = \cS( \Upsilon_{u_{0}}, u_{0}) \text{ for 
all } u_{0} \in M
$$
if the divisor $\cD$ is given in terms of tagged $0$-admissible
Sylvester data sets
as in \eqref{data}.  We now present our solutions of the First and 
Second Interpolation Problems from the Introduction.

To formulate the solution, it is useful to introduce another
definition and some additional notation.  
Given the degree zero null-pole 
divisor $\cD$ as in \eqref{data}, let us say that a divisor $D_{0}$ 
is {\em $\cD$-admissible} if $D_{0}$ is a degree $(g-1)$ divisor of 
the form
 \begin{equation}\label{cD-admis}
  D_{0} = p_{1} + \dots + p_{g} - p_{0}
\end{equation}
where  $D = p_{1} + \cdots + p_{g}$ is nonspecial with distinct
points $p_{0}, \dots, p_{g}$ 
distinct from $u_{1}, \dots, u_{k}$ ($k = n_{\infty} + n_{c} + 
n_{0}$).  The additional notation is:
\begin{align}
& R_{ij}=\text{res}_{p_{j}}[f^{D_{0}}_{u_{i},
A_{\pi_{i}}}(u)C_{\pi_{i}}] \text{ for } i=1,\ldots, n_{P}, \,  
j=1,\ldots, g, \notag \\
&  R=  \begin{bmatrix} 
    R_{11}& \cdots  & R_{1g} \\ 
\vdots &   &  \vdots  \\ 
R_{n_{_{P}}1} &\cdots & R_{n_{_{P}}g} \end{bmatrix}, 
\quad F^{D_{0}}_{A_{\pi}}(u) = {\rm diag.}_{i \in {\rm I} \cup {\rm 
II}} [ f^{D_{0}}_{u_{i}, A_{\pi_{i}}}(u) ], \notag  \\
&  B_{\zeta}= {\rm row}_{j \in {\rm II} \cup {\rm III}} [ 
B_{\zeta_{i}} ], 
\quad  C_{\pi} = {\rm col.}_{j \in {\rm I} \cup {\rm II}} [
C_{\pi_{j}} ].  \label{notation} 
\end{align}

\begin{thm} \label{T:merint}  Let $\cD$ be a degree zero null/pole 
    divisor as in \eqref{data}.  Then:
  \begin{enumerate}
     \item  {\rm \textbf{Solution of the First Interpolation
Problem:}}
   The First Interpolation Problem has a solution, i.e., there exists 
   a (single-valued) meromorphic function $F$ on $M$ which 
   interpolates the divisor $\cD$, if and only if, for any choice of 
   $\cD$-admissible divisor $D_{0}$ as in \eqref{cD-admis}, the 
   the matrix $\Gamma^{D_{0}}_{\cD}$given by \eqref{gamma} is
invertible 
   with inverse $(\Gamma^{D_{0}}_{\cD})^{-1}$ 
    satisfying the side constraint (with notation as in 
    \eqref{notation})
   \begin{equation}  \label{side}
   B_{\zeta} (\Gamma^{D_{0}}_{\cD})^{-1} R = 0.
   \end{equation}
   When this is the case, then the unique interpolant with invertible
value 
   $U_{0}$ at $p_{0}$ is given by
\begin{equation}  \label{autint}  
    K(u) = U_{0} ( I_{r} + B_{\zeta} ( \Gamma^{D_{0}}_{\cD})^{-1}
F^{D_{0}}_{A_{\pi}}(u) C_{\pi})
\end{equation}

\item {\rm \textbf{Solution of the Second Interpolation  Problem:}}  A sufficient 
condition for the Second Interpolation Problem to have a solution, i.e., 
for the existence a relatively automorphic meromorphic matrix function 
$\widehat F$ on $\widehat M$ with a flat factor of automorphy $\zeta_{\widehat F}$ 
which interpolates the null/pole divisor $\cD$, is that there exist a $\cD$-admissible 
divisor $D_{0}$ as in \eqref{cD-admis} so that the matrix $\Gamma^{D_{0}}_{\cD}$ 
given by \eqref{gamma} is invertible.  If $M$ has genus $g=1$, then this sufficient condition is also 
necessary.
\end{enumerate}
\end{thm}

\begin{proof}  Statement (2) is the content of Corollary 7 from 
    \cite{BCV}.  It remains to verify statement (1).

Assume first that there is an $r\times r$-meromorphic
matrix function $G$ on $M$ interpolating the null-pole data
$\mathcal{D}$.
We first note that the existence of such a $G$ is equivalent to the
holomorphic triviality of
the the bundle $E_\mathcal{D}$ and the factor of automorphy
$\zeta_{\mathcal{D}}.$ As a consequence,
$h^{0}(\lambda_{D_{0}}\otimes E_{\mathcal{D}}^{*}) = 0 $.  It follows
from Proposition \ref{P:nonspecialint} that $\Gamma^{D_{0}}_\cD$ is
invertible.

For each $i \in {\rm I} \cup {\rm II}$, near $u_{i}$ ($i \in {\rm I}
\cup {\rm II}$) there is a 
$r\times n_{\pi_{i}}$-
matrix $\widetilde{B}_{i}$ such that 
\begin{equation}   \label{tildeG}
G(u)-\widetilde{B}_{i}((u-u_{i})-A_{\pi_{i}})^{-1}
\end{equation}
is analytic. Let $K$ be the meromorphic matrix function 
$$  
     K = G(p_{0}) + \sum_{i=1}^{n_{P}}\widetilde{B}_{i}
f^{D_{0}}_{w_{i},A_{\pi_{i}}}C_{\pi_{i}},
$$
which is of the form (\ref{matrix}). It follows that the matrix
function $H= G-K$  can only have poles at $p_{1},\ldots, p_{g}$.
Moreover, the divisors $(H_{ij})$ of the entries $H_{ij}$ of $H=G-K$
satisfy 
 $(H_{ij})+D_{0}\geq 0.$ Thus $G=K$.  Since $G$ does not have
poles at the points $p_{1},\ldots, p_{g},$ then
\begin{equation} \label{zeroresidue}
\begin{bmatrix} \widetilde{B}_{1} & \ldots & \widetilde{B}_{n_{P}} 
    \end{bmatrix} R = 0.
\end{equation}
From the fact that $K \in \cO^{r \times r}_{u_i} F_{\cD}$ for $i \in
{\rm 
II} \cup {\rm III}$, application of the criterion \eqref{test} to 
each row of $K$ shows that 
$$
 \begin{bmatrix} \widetilde{B}_{1} &\cdots &  \widetilde{B}_{n_{P}} 
     \end{bmatrix} \Gamma^{D_{0}}_{\cD}=G(p_{0})B_{\zeta}.
$$
As both $\Gamma^{D_{0}}_{\cD}$ and $G(p_{0}) : = U_{0}$ are 
invertible, we may solve uniquely for $\begin{bmatrix} \widetilde 
B_{1} & \cdots & \widetilde B_{n_{P}} \end{bmatrix}$ to get
$$
\begin{bmatrix} \widetilde B_{1} & \cdots & \widetilde B_{n_{P}} 
\end{bmatrix} = U_{0} B_{\zeta} ( \Gamma^{D_{0}}_{\cD})^{-1}.
$$
Substituting this expression for $\widetilde B_{i}$ into 
\eqref{zeroresidue} and \eqref{tildeG} leads us to the validity of 
the side condition \eqref{side} and to the formula \eqref{autint} for
$G$.

We next turn to the sufficiency of the stated conditions. We define 
$K(u)$ by \eqref{autint}.   
Note that $K$ has the form \eqref{matrix} with 
$$
\begin{bmatrix} U_{u_1} & \cdots & U_{u_{n_P}} \end{bmatrix} = U_0 
 \begin{bmatrix} B_{\zeta_1} & \cdots & B_{\zeta_{n_P}} \end{bmatrix}
(\Gamma^{D_0}_\cD)^{-1}.
$$
By Proposition \ref{P:Kexists}, $K$ is the unique meromorphic matrix
function satisfying conditions (1), (2), (3)
 in Proposition \ref{P:Kexists}.  The fact that condition
\eqref{side} is satisfied tells us that $K$ has no poles in 
 $\{p_1, \dots, p_g\}$.  We conclude that condition \eqref{germK}
actually holds at all $u_0 \in M$.  It remains to show that 
 \eqref{germK} actually holds with equality on all of $M$.
 
 To this end we let  $F_{\mathcal{D}}$ be a trivialization of
$\zeta_{\mathcal{D}}$.   
 We now view $K$ as an automorphic meromorphic matrix function on 
all of $\widehat M$.  From condition \eqref{germK} we read off that
the $r \times r$ matrix function $H: = K F_\cD^{-1}$
is holomorphic on all of $\widehat M$.  
Note that $K(p_{0}) =U_{0}$ is invertible, so $\det K$ does not vanish
identically.  As $\det K$ is a single-valued meromorphic matrix
function on $M$,  the winding number of 
$\det K$ around the boundary of the fundamental domain $M_0$ is zero.
The function $F_\cD$ a priori is multivalued when considered as a
function on $M$, but  since it is interpolating the
divisor $\cD$ which has degree equal to 0, it follows that the
winding number of $\det F_\cD$ around the 
boundary of $M_0$ is also zero.  As a consequence,  $\det H$ 
has no zeros on $M_{0}$.  This implies $H$ is invertible on
$\widehat{M}$ and hence
$$
\cO^{r \times r}_{u_{0}} K = \cO^{r \times r}_{u_{0}} H F_{\cD} = 
\cO^{r \times r}_{u_{0}} F_{\cD}
$$
for all $u_{0} \in \widehat M$, i.e., equality holds in \eqref{germK}
for all $u_0$ as
required.  This completes the proof.
\end{proof}

\begin{remark}   \label{R:confession}
{\em The First Interpolation Problem was addressed in \cite{BC} for
the case $g=1$ 
and in \cite{BCV} (see Theorem 9 there) for the case of arbitrary 
genus. The result in \cite{BC} stated the result only for the simple 
multiplicity case; the actual statement was somewhat more cumbersome 
since it was missed there that the matrix $\Gamma^{D_{0}}_{\cD}$ 
necessarily is invertible.  Theorem 9 in \cite{BCV} handles the 
general multiplicity case but the proof there has a gap since it
appears to rely on the
misstatements in Proposition 8 there mentioned above.  The present 
proof uses the corrected version  (Proposition \ref{P:Kexists} above)
of Proposition 8 from \cite{BCV}.}
\end{remark}

\begin{remark} \label{R:HIP/absBV}
   {\em Interpolation problems for relatively automorphic meromorphic 
    matrix functions on Riemann surfaces closely related to those 
    considered here were also studied in \cite{HIP, BV}.  The problem 
    considered in \cite{BV} was as follows:
     \begin{enumerate}
	\item [(IV)] \textbf{Fourth Interpolation Problem:}  Given two 
	flat factors of automorphy $\widetilde \zeta$ and $\zeta$ 
	such that $h^{0}(E_{\widetilde \zeta} \otimes \Delta) = 
	h^{0}(E_{\zeta} \otimes \Delta) = 0$ where $\Delta$ is a
	line bundle (or divisor) of half-order differentials (see 
	Section \ref{S:explicit} below) and 
	given a left matrix null/pole divisor $\cD$ on $M$
	(assumed to have 
	pole and zero order of at most 1 at each point), 
	find a meromorphic matrix function $\widehat G$ on $\widehat 
	M$ so that (i) $\widehat G(Tu) = \widetilde \zeta(T) \widehat 
	G(u) \zeta(T)^{-1}$ for all $u \in \widehat M$ and deck 
	transformations $T$, and (ii) the null/pole structure of $F$ 
	on $M$ is as prescribed by $\cD$.
\end{enumerate}
The interpolation problem considered in \cite{HIP} was the same with 
two modifications:  (1) rather than specifying the input factor of 
automorphy $\zeta$ (or equivalently, the associated bundle 
$E_{\zeta}$), it was only specified that there should be such a 
bundle and part of the problem was to solve also for this bundle, and 
(2)  the formulation was more concrete:  it was assumed that $M$ is 
the normalizing Riemann surface for an algebraic curve ${\mathbf C} = 
\{ \mu \in {\mathbb P}^{w} \colon {\mathbf p}(\mu) = 0\}$ embedded in 
projective space ${\mathbb P}^{2}$ and that the given bundle 
$E_{\widetilde \zeta}$ and the bundle to be found $E_{\zeta}$ are 
presented concretely as kernel bundles associated with determinantal 
representations of the polynomial ${\mathbf p}^{r}$ ($r$ equal to the rank of 
the $E_{\widetilde \zeta}$ and $E_{\zeta}$):
\begin{align*}
    &  E_{\widetilde \zeta} = \{( (\mu, u) \colon \mu \in {\mathbf C}, 
     \, u \in {\mathbb C}^{M} \colon (\mu_{2} \widetilde \sigma_{1} + \mu_{1} 
     \widetilde \sigma_{2} + \mu_{0} \widetilde \gamma) u = 0\}, \\
     & E_{\zeta} = \{( (\mu, u) \colon \mu \in {\mathbf C}, 
     \, u \in {\mathbb C}^{M} \colon (\mu_{2} \sigma_{1} + \mu_{1} 
     \sigma_{2} + \mu_{0} \gamma) u = 0\}
\end{align*}
where
$$
 {\mathbf p}(\mu)^{r} = \det( \mu_{2} \widetilde \sigma_{1} + \mu_{1} 
     \widetilde \sigma_{2} + \mu_{0} \widetilde \gamma) =
     \det (\mu_{2} \sigma_{1} + \mu_{1} 
     \sigma_{2} + \mu_{0} \gamma).
$$
The precise connection between the Fourth Interpolation Problem and 
the variant considered in \cite{HIP} is explained in some detail in 
Section 6 of \cite{BV}

To compare Problems (IV) and (I), we identify a special case of 
Problem (IV) which can be related to a special case of Problem (I) 
as follows.  We first note that one point of incompatibility between 
the two problems is that Problem (IV) is formulated in terms of a 
left matrix null/pole divisor while Problem (I) is formulated in 
terms of a right matrix null/pole divisor.  However, if $\cD$ as in 
\eqref{data} and \eqref{localdata} is a right matrix null/pole divisor, then
$\cD' = \{ \Upsilon_{q_{0}} \colon q_{0} \in M_{0}\}$ with
$$
\Upsilon_{q_{0}}' = \{ (C_{\pi_{q_{0}}}, A_{\pi_{q_{0}}}), 
  (A_{\zeta_{q_{0}}}, B_{\zeta_{q_{0}}}), - S_{q_{0}})
$$
is a left matrix null/pole divisor and if the meromorphic matrix 
function $\widehat F$ has right null/pole structure fitting $\cD$ 
if and only if $G = F^{-1}$ has left null/pole structure fitting 
$\cD'$.  Thus we may reformulate Problem (IV) in terms of $F = 
G^{-1}$ rather than $F$: then we must have an $F$ with prescribed 
right null/pole structure prescribed by the right matrix null/pole 
divisor $\cD$ over $M_{0}$ which in addition has the relative 
automorphy property
\begin{equation}  \label{matrixrelauto}
   F(Tu) =  \zeta(T) \widehat F(u) \widetilde \zeta(T)^{-1}.
\end{equation}

If we also insist that $\zeta = \widetilde \zeta = \zeta_{0} \otimes 
I_{r}$ for a flat scalar factor of automorphy $\zeta_{0}$, then the relative 
automorphy property \eqref{matrixrelauto} imposed on $F$ just means 
that $F$ is automorphic:  $F(Tu) = F(u)$ for all deck transformations 
$T$. Let $D = p_{1} + \cdots + 
p_{g}$ be a nonspecial divisor and choose $\zeta_{0}$ to be the flat 
line bundle so that so that 
$\zeta_{0} \otimes \Delta = \lambda_{D_{0}}$ where $D_{0} = D - 
p_{0}$.  Then the reformulation of 
Problem (IV) becomes exactly Problem (I).  
The one additional point is that the work in \cite{BV} was only carried out for the 
case where the divisor $\cD $ has pole and zero order of at most 1 at each point.
To compare solutions, for simplicity we shall also insist that poles 
and zeros are disjoint and of multiplicity 1.

We therefore assume that there are points $\mu_{1}, \dots, \mu_{N} \in 
M$ (the poles) and $\lambda_{1}, \dots, \lambda_{N} \in M$ 
(the zeros), all distinct, along with a specified nonzero $1 \times r$ 
row vector $u_{j}$ (the left pole vector at $\mu_{j}$) and  a nonzero 
$r \times 1$ column vector $x_{j}$ (the right null vector at $\lambda_{j}$) 
($j = 1, \dots, N$) so that
\begin{equation}   \label{BabyDivisor}
\Upsilon_{q_{0}} = \begin{cases} ((x_{j},0), (\emptyset, \emptyset), 
\emptyset) & \text{ if } q_{0} = \mu_{j}, \\
((\emptyset, \emptyset), (0, u_{i}), \emptyset) & \text{ if } q_{0} 
= \lambda_{i}, \\
 \emptyset & \text{ otherwise.}
\end{cases}
\end{equation} 
Then the solution criterion from \cite[Theorem 3.1]{BV} (after transcription from 
right to left formulation as explained above) is: {\em a (necessarily 
unique) solution exists if and only if
\begin{equation}   \label{Gamma0}
 \Gamma^{0} = \left[ \Gamma^{0}_{ij} \right]_{i,j=1}^{N} = 
 \left[ -u_{i}\left( K (\zeta_{0};  \mu_{i}, \lambda_{j}) \otimes I_{r} \right)
 x_{j} \right]_{i,j=1}^{N}
\end{equation}
is invertible, together with a linear side-constraint to guarantee 
that the solution has no poles at the points $p_{1}, \dots, p_{g}$.} 
Here $K(\zeta_{0}; \cdot, \cdot)$ is the Cauchy kernel associated 
with the flat factor of automorphy $\zeta_{0}$ (see the appendix Section 
\ref{S:explicit} below for a brief introduction to this Cauchy 
kernel).

We note that the same problem has a solution via statement (1) in 
Theorem \ref{T:merint}: {\em a (necessarily unique) solution exists if 
and only if the matrix 
$$ 
  \Gamma^{D_{0}}_{\cD} =[ -{\rm res}_{p = \lambda_{j}} 
  f^{D_{0}}_{\mu_{i}}(p) x_{i} u_{j} (z_{\lambda_{j}})^{-1}]_{i,j=1}^{N}
$$
is invertible, together with a linear side-constraint to guarantee 
that the solution has no poles at the points $p_{1}, \dots, p_{g}$.}
By our assumptions that poles and zeros are distinct,
$f^{D_{0}}_{\mu_{i}}$ is analytic at $\lambda_{j}$.  From the formula 
\eqref{fD0wfor} explained in Theorem \ref{T:fw-form} below, there is 
a connection between the building-block functions $f^{D_{0}}_{\mu}$ 
and the Cauchy kernel, namely: 
 $$
 f^{D_{0}}_{\mu_{i}}(\lambda_{j}) = \frac{K(\zeta_{0}; 
\lambda_{j}, \mu_{i})}{K(\zeta_{0}; \lambda_{j}, p_{0})} K(\zeta_{0}; 
\mu_{i},p_{0}). 
$$
Trivially ${\rm res}_{p = \lambda_{j}} 
(z_{\lambda_{j}})^{-1} = 1$ (where $z_{\lambda_{j}}$ is the local 
coordinate on $M$ centered at $\lambda_{j}$).  Hence the formula for 
$\Gamma^{D_{0}}_{\cD}$ in this case becomes
\begin{equation}   \label{BabyGammaD0cD}    
   \Gamma^{D_{0}}_{\cD} = \left[ - f^{D_{0}}_{\mu_{i}}(\lambda_{j}) x_{i} u_{j}
  \right]_{i,j=1}^{N}.
=  \left[-  \frac{K(\zeta_{0}; 
\lambda_{j}, \mu_{i})}{K(\zeta_{0}; \lambda_{j}, p_{0})} K(\zeta_{0}; 
\mu_{i},p_{0}) x_{i} u_{j} \right]_{i,j = 1}^{N}. 
\end{equation}
We note that the matrix $\Gamma^{D_{0}}_{\cD}$ is equal to the matrix 
$\Gamma^{0}$ \eqref{Gamma0} multiplied on the left and on the right 
by invertible diagonal matrices, i.e., invertibility of $\Gamma^{0}$ 
is equivalent to invertibility of $\Gamma^{D_{0}}_{\cD}$.  In this 
way we see directly the equivalence of the solutions of this special 
interpolation problem as given by Theorem \ref{T:merint} and as given 
by Theorem 3.1 in \cite{BV}. 
}\end{remark}

\begin{remark} {\em \textbf{Abel's theorem.} 
Let us specialize the setting of Remark \ref{R:HIP/absBV} even further by assuming that 
$r=1$, i.e., we wish to solve for a scalar meromorphic function (or 
more generally, relatively automorphic function with flat factor of 
automorphy) on $\widehat M$ with prescribed distinct simple zeros $\lambda_{1}, \dots, \lambda_{N}$ and 
prescribed distinct simple poles $\mu_{1}, \dots, \mu_{N}$ in $M$.  
We therefore assume that the divisor $\cD$ is given as $\cD = \{ 
\Upsilon_{q_{0}} \colon q_{0} \in M\}$ where $\Upsilon_{q_{0}}$ is 
given as in \eqref{BabyDivisor} with each $u_{j}$ and $x_{i}$ taken 
to be the complex number 1.  When combined with formula 
\eqref{Cauchykernelform2} from the Appendix for the Cauchy kernel, we 
see that the formula \eqref{BabyGammaD0cD} 
for $\Gamma^{D_{0}}_{\cD}$ simplifies further to
\begin{align}  
   & \Gamma^{D_{0}}_{\cD} = \left[ -f_{\mu_{i}}^{D_{0}}(\lambda_{j}) 
    \right]_{i,j=1}^{N} = \left[-  \frac{K(\zeta_{0}; \lambda_{j}, \mu_{i})  
    K(\zeta_{0}; \mu_{i},p_{0}) }     
{K(\zeta_{0}; \lambda_{j}, p_{0})} \right]_{i,j = 1}^{N}  
\label{scalarBaby}\\
& =\left[ -\frac{1}{\theta(\be)}  
\frac{\theta(\phi(\mu_{i}) - \phi(\lambda_{j}) + \be)}{E_{\Delta}(\mu_{i}, \lambda_{j})} 
\frac{E_{\Delta}(p_{0}, \lambda_{j})}{\theta(\phi(p_{0}) -\phi(\lambda_{j}) + \be)}
\frac{\theta(\phi(p_{0}) - \phi(\mu_{i}) + \be)}{E_{\Delta}(p_{0}, 
\mu_{i})}\right]
\notag 
\end{align}
We note that our Theorem \ref{T:merint} part (2) gives invertibility 
of $\Gamma^{D_{0}}_{\cD}$ for some $\cD$-admissible divisor $D_{0}$ 
as a sufficient condition for the existence of a flat relatively 
automorphic solution of the interpolation problem. 

On the other hand,
it is well known for this scalar version of the problem that, 
given such a scalar divisor $D = \lambda_{1} + \cdots + \lambda_{N} - 
\mu_{1} - \cdots - \mu_{N}$, there always exists a function $f_{D}$ on 
$\widehat M$ with flat factor of automorphy interpolating the data $D$ 
on $M$.  Such a function can be constructed using the prime form:
\begin{equation}   \label{primeformsol}
    f(p) = \frac{ \prod_{j=1}^{N} E_{\Delta}(p, \lambda_{j})}
    { \prod_{i=1}^{N} E_{\Delta}(p, \mu_{i})}.
\end{equation}
The construction in the single-valued case where 
$\sum_{i=1}^{N} \phi(\mu_{i}) = \sum_{j=1}^{N} \phi(\lambda_{j})$ mod the period lattice is carried 
out in Mumford's book\cite[pages 3.209--3.212]{MumfordII}; the reader can check 
that more generally the factor of automorphy in \eqref{primeformsol} is flat as long as 
the number of prescribed zeros is equal to the number of prescribed poles (counting multiplicities).

This stronger result for the scalar case can be explained as follows.
From the last formula for $\Gamma_{\cD}^{D_{0}}$ in 
\eqref{scalarBaby}, we see the factorization for $\Gamma_{\cD}^{D_{0}}$:
\begin{equation}   \label{Gamma-alt}
    \Gamma_{\cD}^{D_{0}} = -\frac{1}{\theta(\be)} \cdot \bD_{\bmu} \cdot \bM \cdot \bD_{\blam}
\end{equation}
where
\begin{align*}
    & \bD_{\bmu} = {\rm diag}_{1 \le i \le N} \left[ \frac{ 
    \theta(\phi(p_{0}) - \phi(\mu_{i}) + \be)}
    {E_{\Delta}(p_{0}, \mu_{i})} \right],  \\
    & \bM = \left[ \frac{\theta(\phi(\mu_{i}) - \phi(\lambda_{j}) + \be)}
    {E_{\Delta}(\mu_{i}, \lambda_{j})} \right]_{i,j=1}^{N},  \\
& \bD_{\blam} = {\rm diag}_{1 \le j \le N}\left[ \frac{E_{\Delta}(p_{0}, 
\lambda_{j})}{ \theta(\phi(p_{0}) - \phi(\lambda_{j}) + \be)}\right].
\end{align*}
As $\bD_{\bmu}$ and $\bD_{\blam}$ are invertible diagonal matrices, we 
see that $\Gamma_{\cD}^{D_{0}}$ is invertible if and only if the 
middle factor $\bM$ is invertible.  It turns out the $\det \bM$ can be 
computed explicitly (see Corollary 2.19 page 33 in Fay's book 
\cite{Fay}):
\begin{align}  
&  \det \bM  = \notag \\
&   \theta \left( \sum_{i} \phi(\mu_{i}) - \sum_{j} \phi(\lambda_{j}) 
+ \be \right) \theta(\be)^{N-1}
\frac{ \prod_{i<j} E_{\Delta}(\mu_{i}, \mu_{j}) 
E_{\Delta}(\lambda_{j}, \lambda_{i})}
{\prod_{i,j} E_{\Delta}(\mu_{i}, \lambda_{j})}.
\label{Fay}
\end{align}
Since $\theta(\be) \ne 0$ and the last factor is automatically nonzero by the assumption 
that the points $\mu_{1}, \dots, \mu_{N}, \lambda_{1}, \dots, 
\lambda_{N}$ are all distinct, we see that the criterion for $\det 
\Gamma^{D_{0}}_{\cD} \ne 0$ is given by
\begin{equation}   \label{Faycriterion}
 \theta\left(\sum_{i=1}^{N} \phi(\mu_{i}) - \sum_{j=1}^{N} 
 \phi(\lambda_{j}) + \be \right) \ne 0
\end{equation}
where(by \eqref{be} $\be = a \Omega + b$.
By inspection we see that condition \eqref{Faycriterion} holds for a 
generic choice of  line bundle $\zeta$.  We conclude that in the scalar case,
part (2) of Theorem \ref{T:merint} can be strengthened to: {\em the matrix 
$\Gamma^{D_{0}}_{\cD}$ is invertible for a generic choice of 
$\cD$-admissible divisor $D_{0}$ and the Second Interpolation Problem 
is always solvable.}

Unlike the classical solution in terms of prime forms from 
\cite{MumfordI}, our solution of the Second Interpolation Problem for 
the simple-multiplicity scalar case goes through Weil's 
characterization of flat bundles and is not constructive.  The 
formula \eqref{autint} used to solve the First Interpolation Problem 
for this case gives an automorphic (single-valued) meromorphic 
function interpolating the divisor $\cD$ but carrying possible extra 
poles in $\{p_{1}, \dots, p_{g}\}$ along with compensating additional 
zeros at undetermined points.  The side constraint \eqref{side} 
removes any extra poles in $\{p_{1}, \dots, p_{g}\}$ and thereby 
generates a single-valued solution of the interpolation problem (i.e.,
a solution of the First Interpolation Problem).
The linear side constraint \eqref{side}, when spelled out for this 
case, becomes
\begin{equation}   \label{BCVsol}
    \sum_{i,j=1}^{N} \left[ (\Gamma^{D_{0}}_{\cD})^{-1}\right]_{ij} 
    \frac{K(\zeta_{0}; p_{k}, \mu_{j}) K(\zeta_{0}; \mu_{j},p_{0})}
 {\frac{d}{dp}|_{p=p_{k}} \{ K(\zeta_{0}; p, p_{0})\} } = 0 \text{ 
 for } k=1, \dots, g
\end{equation}
with $\Gamma^{D_{0}}_{\cD}$ as in \eqref{scalarBaby}.

The entries of $(\Gamma_{\cD}^{D_{0}})^{-1}$ can be spelled out more 
explicitly as follows.  From the factorization \eqref{Gamma-alt} for 
$\Gamma_{\cD}^{D_{0}}$, we get the factorization for 
$(\Gamma_{\cD}^{D_{0}})^{-1}$:
\begin{equation}  \label{Gammainvfact}
  (\Gamma_{\cD}^{D_{0}})^{-1} = -\theta(\be) \cdot \bD_{\blam}^{-1} 
  \cdot \bM^{-1} \cdot \bD_{\bmu}^{-1}.
\end{equation}
We note that each minor of $\bM$, $\bM_{ij}$ (the $(N - 1) 
\times (N-1)$ submatrix of $\bM$ formed by crossing out row $i$ and 
column $j$),  has the same form as $M$; the only adjustment is that 
the set of poles is one less with $\mu_{i}$ omitted and the set of 
zeros is one less with $\lambda_{j}$ omitted.  One can therefore 
compute the entries  $C_{\alpha \beta} = (-1)^{\alpha + \beta} \det 
\bM_{\beta \alpha}$ by another application of Fay's identity; the 
result is 
\begin{align}
  C_{\alpha \beta} = & (-1)^{\alpha + \beta}\,  \theta(\sum_{j \ne 
  \beta} \phi(\mu_{j}) - \sum_{i \ne \alpha} \phi(\lambda_{i}) + \be)
 \cdot  \theta(\be)^{N-2} \notag\\
 & \quad \cdot \frac{ \prod_{i<j \colon i,j \ne \beta} E_{\Delta}(\mu_{i}, \mu_{j}) 
  \prod_{i<j \colon i,j \ne \alpha} E_{\Delta}(\lambda_{j}, \lambda_{i})}
  { \prod_{i,j \colon i \ne \alpha, j \ne \beta} E_{\Delta}(\mu_{j}, 
  \lambda_{i})}
  \label{cofac}
\end{align}
By Cramer's Rule, the entries $(M^{-1})_{\alpha \beta}$ ($1 \le 
\alpha, \beta \le N$) of $M^{-1}$ are given by 
$(\bM^{-1})_{\alpha \beta} = \frac{ C_{\alpha \beta}}{\det \bM}$.  
We then compute
\begin{align}
    & (\bM^{-1})_{\alpha \beta} = \frac{  C_{\alpha 
    \beta}}{\det \bM}  = (-1)^{\alpha + \beta} \frac{1}{\theta(\be)} \cdot 
\frac{\theta( \sum_{j \ne \beta} \phi(\mu_{j}) - \sum_{i \ne \alpha} 
\phi(\lambda_{i})  + \be)}
{\theta(\sum_{j} \phi(\mu_{j}) - \sum_{i} \phi(\lambda_{i}) + \be)} \cdot 
\notag \\
& \cdot \frac{ \prod_{i<j \colon i,j \ne \beta} 
E_{\Delta}(\mu_{i}, \mu_{j})}
 {\prod_{i<j} E_{\Delta}(\mu_{i}, \mu_{j})}  \cdot
\frac{ \prod_{i<j \colon i,j \ne \alpha} 
E_{\Delta}(\lambda_{j}, \lambda_{i})}
 {\prod_{i<j} E_{\Delta}(\lambda_{j}, \lambda_{i})}
 \cdot \frac{ \prod_{i,j} E_{\Delta}(\mu_{j}, \lambda_{i})}
{ \prod_{i,j \colon i \ne \alpha, j \ne \beta} E_{\Delta}(\mu_{j}, 
\lambda_{i})}. \notag  
\end{align}
Noting the cancellations in the prime-form terms then compactifies 
this expression to
\begin{align}
& (\bM^{-1})_{\alpha \beta} = (-1)^{\alpha + \beta} \frac{1}{\theta(\be)} \cdot 
\frac{\theta( \sum_{j \ne \beta} \phi(\mu_{j}) - \sum_{i \ne \alpha} 
\phi(\lambda_{i})  + \be)}
{\theta(\sum_{j} \phi(\mu_{j}) - \sum_{i} \phi(\lambda_{i}) + \be)}
\cdot  \notag \\ 
& \quad \cdot \frac{1}{ \prod_{i< \beta} E_{\Delta}(\mu_{i}, \mu_{\beta})
\cdot \prod_{\beta < j} E_{\Delta}(\mu_{\beta}, \mu_{j})} \cdot
\frac{1}{ \prod_{i< \alpha} E_{\Delta}(\lambda_{\alpha}, \lambda_{i})
\cdot \prod_{\alpha < j} E_{\Delta}(\lambda_{j}, \lambda_{\alpha})} 
\cdot  \notag \\
& \quad \cdot  \prod_{i} E_{\Delta}(\mu_{\beta}, \lambda_{i}) \cdot
\prod_{j} E_{\Delta}(\mu_{j}, \lambda_{\alpha}) \cdot 
\frac{1}{E(\mu_{\beta}, \lambda_{\alpha})}  \notag \\
& = \frac{1}{\theta(\be)} \cdot 
\frac{\theta( \sum_{j \ne \beta} \phi(\mu_{j}) - \sum_{i \ne \alpha} 
\phi(\lambda_{i})  + \be)}
{\theta(\sum_{j} \phi(\mu_{j}) - \sum_{i} \phi(\lambda_{i}) + \be)}
\cdot \notag \\
& \quad \cdot \frac{ \prod_{i}E_{\Delta}(\mu_{\beta}, \lambda_{i}) \cdot
\prod_{j} E_{\Delta}(\mu_{j}, \lambda_{\alpha})}
{ \prod_{j \ne \beta} E_{\Delta}(\mu_{j}, \mu_{\beta}) \cdot
\prod_{i \ne \alpha} E_{\Delta}(\lambda_{\alpha}, \lambda_{i})
\cdot E(\mu_{\beta}, \lambda_{\alpha})}
 \label{Minverse}
\end{align}
where we used the prime-form property $E_{\Delta}(x, y) = - 
E_{\Delta}(y,x)$ in the last step.
From the factorization \eqref{Gammainvfact} we get the still 
 longer formula for $(\Gamma_{\cD}^{D_{0}})^{-1}$:
\begin{align}
  &  [(\Gamma_{\cD}^{D_{0}})^{-1}]_{\alpha \beta} =  \label{Gammainverse} \\
  & -\frac{\theta(\phi(p_{0}) - \phi(\lambda_{\alpha}) + \be)}
  { E_{\Delta}(p_{0}, \lambda_{\alpha})}\cdot \frac{ 
  \theta(\sum_{j \ne \beta} 
  \phi(\mu_{j}) - \sum_{i \ne \alpha} \phi(\lambda_{i}) + \be)}
 {\theta(\sum_{j} \phi(\mu_{j}) - \sum_{i} \phi(\lambda_{i}) + \be)} 
 \cdot \notag \\
 &\cdot \frac{ \prod_{j} E_{\Delta}(\mu_{j}, \lambda_{\alpha}) 
 \cdot \prod_{i} E_{\Delta}(\mu_{\beta}, \lambda_{i})}
 {\prod_{j \ne \beta} E_{\Delta}(\mu_{j}, \mu_{\beta}) \cdot \prod_{i \ne 
 \alpha} E_{\Delta}(\lambda_{i}, \lambda_{\alpha})} \cdot 
 \frac{1}{E_{\Delta}(\mu_{\beta}, \lambda_{\alpha})} \cdot
\frac{E_{\Delta}(p_{0}, \mu_{\beta})}{ \theta(\phi(p_{0}) - \phi(\mu_{\beta}) + \be)}.
\notag
\end{align}

On the other 
hand, Abel's Theorem (see e.g.\ \cite[page 97]{FarkasKra}, 
\cite[pages 145--160]{MumfordI}, or \cite[Chapter 8]{Miranda}) tells us 
that this scalar null/pole interpolation problem has a solution 
exactly when
\begin{equation}  \label{Abel'}
    \sum_{i=1}^{N} \phi(\mu_{i}) = \sum_{j=1}^{N} \phi(\lambda_{j})
    + m + n \Omega \text{ for some } m, n \in {\mathbb Z}^{g}
\end{equation}
where $\phi$ is the Abel-Jacobi map \eqref{AbelJacobi}; indeed, 
this is just the condition to force the prime-form solution 
\eqref{primeformsol} to be single-valued. It follows 
that the side condition \eqref{BCVsol} must be equivalent to the Abel 
condition \eqref{Abel}.

Assume that the Abel condition \eqref{Abel} holds.  For purposes of 
computation, we can view all functions as being defined on the 
universal cover of $M$ and choose a  new divisor $\widetilde 
\lambda_{1} + \cdots \widetilde \lambda_{N} -\widetilde \mu_{1}, 
\dots, \widetilde \mu_{N}$ sitting above $\lambda_{1} + \cdots + 
\lambda_{n} - \mu_{1} - \cdots - \mu_{N}$ on the universal cover so 
that we have the stronger version $\sum_{i=1}^{N} \phi(\widetilde 
\mu_{i}) = \sum_{j=1}^{N} \phi(\widetilde \lambda_{j})$ of the Abel 
condition \eqref{Abel'}.  
In the sequel we shall assume that this 
normalization has been done so that we can assume the stronger form 
of the Abel condition:
\begin{equation}  \label{Abel}
    \sum_{i=1}^{N} \phi(\mu_{i}) = \sum_{j=1}^{N} \phi(\lambda_{j}).
\end{equation}
Then it is immediate from the Fay criterion 
\eqref{Faycriterion} (in fact, even without the normalization) that $\det \bM$ and hence also  
$\det \Gamma^{D_{0}}_{\cD}$ are nonzero, since in this case
$$
\theta(\sum_{i} \phi(\mu_{i}) - \sum_{j} \phi(\lambda_{j}) + \be)
= \theta( \be).
$$
Then we have two formulas for the zero-pole interpolant, namely 
\eqref{primeformsol} and \eqref{autint}.  If $U_{0}$ in 
\eqref{autint} is chosen to match the value of the first solution at 
the point $p_{0}$, these two formulas must yield the same function.
For the simple-multiplicity scalar-valued case which we are 
discussing here, the formula \eqref{autint} can be made more explicit 
as follows.  Recall that $B_{\zeta} = \begin{bmatrix} 1 & \cdots & 1 
\end{bmatrix}$,  $C_{\pi} = \sbm{ 1 \\ \vdots \\ 1}$ and the matrix 
$F^{D_{\zeta_{0}}}$ has a theta function representation by formula 
\eqref{fD0wfor}: 
$$
F^{D_{\zeta_{0}}}_{A_{\pi}}(p) = {\rm diag}_{1 \le j \le N} \left[ 
f^{D_{\zeta_{0}}}_{\mu_{j}}(p) \right].
$$
We then arrive at the conclusion: {\em  if the Abel condition \eqref{Abel} 
is satisfied, then there is a nonzero complex number $K$ so that the 
following identity for all $p \in M \setminus \{\mu_{1}, \dots, 
\mu_{N}\}$:}
\begin{equation}   \label{intid}
   1 + \sum_{i,j=1}^{N} \left[ 
    (\Gamma_{\cD}^{D_{\zeta_{0}}})^{-1}\right]_{ij} 
    f^{D_{\zeta_{0}}}_{\mu_{j}}(p)
   = K  \frac{ \prod_{j=1}^{N} E_{\Delta}(p,\lambda_{j})}
{\prod_{i=1}^{N} E_{\Delta}(p,\mu_{i})}
\end{equation}
where the value of the constant $K$ is necessarily given by
\begin{equation}   \label{intid'}
  K = \frac{ \prod_{i=1}^{N} E_{\Delta}(p_{0}, \mu_{i})}
  {\prod_{j=1}^{N} E_{\Delta}(p_{0},\lambda_{j})}
\end{equation}
in order that the two sides of this expression agree at the base 
point $p_{0}$.   We now give an independent direct computational verification that 
the identity \eqref{intid} indeed does hold under the assumption that 
the Abel condition \eqref{Abel} holds as follows.

As a first observation, we note that:  {\em to show that 
\eqref{intid}--\eqref{intid'} holds, it suffices to show the equality 
of residues
\begin{align}  
  &  {\rm res}_{p = \mu_{\beta}} \left(1 + \sum_{i,j=1}^{N} \left[ 
    (\Gamma_{\cD}^{D_{\zeta_{0}}})^{-1}\right]_{ij} 
    f^{D_{\zeta_{0}}}_{\mu_{j}}(p)\right)  \notag \\
 & \quad =
    {\rm res}_{p = \mu_{\beta}} 
\frac{ \prod_{i=1}^{N} E_{\Delta}(p_{0}, \mu_{i})}
  {\prod_{j=1}^{N} E_{\Delta}(p_{0},\lambda_{j})}
  \frac{ \prod_{j=1}^{N} E_{\Delta}(p,\lambda_{j})}
{\prod_{i=1}^{N} E_{\Delta}(p,\mu_{i})}
 \label{res=}
\end{align}
for each $\beta = 1, \dots, N$.}   Indeed, suppose that \eqref{res=} 
for all $\beta$ and set 
$$
h(p) = 1 + \sum_{i,j=1}^{N} \left[ 
    (\Gamma_{\cD}^{D_{\zeta_{0}}})^{-1}\right]_{ij} 
    f^{D_{\zeta_{0}}}_{\mu_{j}}(p) - 
\frac{ \prod_{i=1}^{N} E_{\Delta}(p_{0}, \mu_{i})}
  {\prod_{j=1}^{N} E_{\Delta}(p_{0},\lambda_{j})}
  \frac{ \prod_{j=1}^{N} E_{\Delta}(p,\lambda_{j})}
{\prod_{i=1}^{N} E_{\Delta}(p,\mu_{i})}.
$$
Then $h$ is a (single-valued) meromorphic function on $M$ which has a 
zero at $p_{0}$ and only possible poles equal to at most simple poles 
at $p_{1}, \dots, p_{g}$, i.e., $g$ is a holomorphic section of the 
bundle associated with the divisor $D - \{p_{0}\}$.  Since $D$ by 
assumption is nonspecial, it follows that $h \equiv 0$, i.e., 
\eqref{intid}--\eqref{intid'} holds.  As the right-hand side of 
\eqref{intid} has no poles at the points $p_{1}, \dots, p_{g}$, it 
follows that the apparent possible poles at $p_{1}, \dots, p_{g}$ of 
the left-hand expression are all removable, from which the side 
condition \eqref{BCVsol} follows as well.

We now assume that the strong Abel condition \eqref{Abel} holds.  Our goal is to show the 
residue equality \eqref{res=}.

We first note that the second factor in the formula 
\eqref{Gammainverse} simplifies when \eqref{Abel} holds, namely:
$$
\frac{ \theta(\sum_{j \ne \beta} \phi(\mu_{j}) - \sum_{i \ne \alpha} \phi(\lambda_{i}) + \be)}
 {\theta(\sum_{j} \phi(\mu_{j}) - \sum_{i} \phi(\lambda_{i}) + \be)} 
= \frac{\theta(\phi(\lambda_{\alpha}) - 
\phi(\mu_{\beta}) + \be)}{  \theta(\be)}.
$$
Thus the expression \eqref{Gammainverse} for $[ 
(\Gamma^{D_{0}}_{\cD})^{-1}]_{\alpha \beta}$ simplifies to
\begin{align*}
  &  [(\Gamma_{\cD}^{D_{0}})^{-1}]_{\alpha \beta}   \\
 & = - \frac{\theta(\phi(p_{0}) - \phi(\lambda_{\alpha}) + 
  \be)}{ E_{\Delta}(p_{0}, \lambda_{\alpha})} \cdot
  \frac{\theta(\phi(\lambda_{\alpha}) - \phi(\mu_{\beta}) + 
  \be)}{\theta(\be) E(\mu_{\beta}, \lambda_{\alpha})}
  \cdot \frac{E_{\Delta}(p_{0}, \mu_{\beta})}{\theta(\phi(p_{0}) - 
  \phi(\mu_{\beta}) + \be)} \cdot \\
 & \quad \quad \quad \cdot  \frac{\prod_{j} E_{\Delta}(\mu_{j}, \lambda_{\alpha}) 
\cdot \prod_{i} E_{\Delta}(\mu_{\beta}, \lambda_{i})}
{\prod_{j \ne \beta} E_{\Delta}(\mu_{j}, \mu_{\beta}) \cdot \prod_{i  
\ne \alpha} E_{\Delta}(\lambda_{i}, \lambda_{\alpha})} \\
 & = - f^{D_{0}}_{\lambda_{\alpha}}(\mu_{\beta}) \cdot 
 \frac{\prod_{j} E_{\Delta}(\mu_{j}, \lambda_{\alpha}) 
\cdot \prod_{i} E_{\Delta}(\mu_{\beta}, \lambda_{i})}
{\prod_{j \ne \beta} E_{\Delta}(\mu_{j}, \mu_{\beta}) \cdot \prod_{i 
\ne \alpha} E_{\Delta}(\lambda_{i}, \lambda_{\alpha})}.
\end{align*}
where we use the identity \eqref{fD0wfor'} as well as $E(\mu_{\beta}, 
\lambda_{\alpha}) = -E_{\Delta}(\lambda_{\alpha}, \mu_{\beta})$ for 
the last step.
  
The problem of verifying \eqref{res=} therefore comes down to 
verifying
\begin{equation}   \label{verify1}
 \sum_{\alpha = 1}^{N}  [ (\Gamma_{\cD}^{D_{0}})^{-1}]_{\alpha \beta} 
=  \frac{ \prod_{j=1}^{N} E_{\Delta}(p_{0}, \mu_{j})}
  {\prod_{i=1}^{N} E_{\Delta}(p_{0},\lambda_{i})} \cdot
  \frac{ \prod_{i}E_{\Delta}(\mu_{\beta}, \lambda_{i})}{ \prod_{j \ne \beta} 
  E_{\Delta}(\mu_{\beta}, \mu_{j})}
 \end{equation}
  for all $\beta = 1, \dots, N$, 
  Let us rewrite the right-hand side of \eqref{verify1} as
  $$
   \frac{ \prod_{j=1}^{N} E_{\Delta}(p_{0}, \mu_{j})}
  {\prod_{i=1}^{N} E_{\Delta}(p_{0},\lambda_{i})} \cdot
  \frac{ \prod_{i}E_{\Delta}(\mu_{\beta}, \lambda_{i})}{ \prod_{j \ne \beta} 
  E_{\Delta}(\mu_{\beta}, \mu_{j})} =-
  \frac{ \prod_{j=1}^{N} E_{\Delta}(\mu_{j}, p_{0})}{\prod_{i=1}^{N} 
  E_{\Delta}(p_{0}, \lambda_{i})}
  \cdot \frac{\prod_{i}E_{\Delta} (\mu_{\beta}, \lambda_{i})}
  {\prod_{j \ne \beta} E_{\Delta}(\mu_{j}, \mu_{\beta})}.
 $$
 The left-hand side of \eqref{verify1} is given by
\begin{align*}
     \sum_{\alpha = 1}^{N}  [ (\Gamma_{\cD}^{D_{0}})^{-1}]_{\alpha 
     \beta} =  \sum_{\alpha = 1}^{N} 
-f^{D_{0}}_{\lambda_{\alpha}}(\mu_{\beta}) \cdot 
 \frac{\prod_{j} E_{\Delta}(\mu_{j}, \lambda_{\alpha}) 
\cdot \prod_{i} E_{\Delta}(\mu_{\beta}, \lambda_{i})}
{\prod_{j \ne \beta} E_{\Delta}(\mu_{j}, \mu_{\beta}) \cdot 
\prod_{i \ne \alpha} E_{\Delta}(\lambda_{i}, \lambda_{\alpha})}.
\end{align*}
Cancellation of the common factor $\frac{ \prod_{i} 
E_{\Delta}(\mu_{\beta}, \lambda_{i})}{ \prod_{j \ne \beta} 
E_{\Delta}(\mu_{\beta}, \mu_{j})}$ and writing 
$f_{\lambda_{\alpha}}^{D_{0}}$ in the more explicit form
$f_{\lambda_{\alpha}}^{D - p_{0}}$ then
converts the desired identity \eqref{verify1} to 
\begin{align}
 \sum_{\alpha=1}^{N} 
 f_{\lambda_{\alpha}}^{D- p_{0}}(\mu_{\beta}) \frac{ \prod_{j} 
 E_{\Delta}(\mu_{j}, \lambda_{\alpha})}{\prod_{i \ne \alpha} 
 E_{\Delta}(\lambda_{i}, \lambda_{\alpha})} 
  = \frac{ \prod_{j} E_{\Delta}(\mu_{j}, p_{0})}
  {\prod_{i} E_{\Delta}(p_{0},\lambda_{i})}.
\label{intid2}
\end{align}
We now view $f_{\lambda_{\alpha}}^{D - p_{0}}(\mu_{\beta})$ as a function of $p_{0}$ 
for each fixed $\lambda_{\alpha}$ and $\mu_{\beta}$.  
As noted in Corollary \ref{C:pvsp0} below, when viewed as a (single-valued) meromorphic 
function on $M$ as a 
function of $p_{0}$, $f_{\lambda_{\alpha}}^{D - p_{0}}(\mu_{\beta})$ has a  simple pole 
at $\lambda_{\alpha}$ with residue equal  to $-1$ together with a 
zero at $\mu_{\beta}$, and other possible 
poles at most $g$ in number in the zero divisor of $K(\zeta_{0}; 
\mu_{\beta}, \cdot)$ on $M$.  For generic choice of $\mu_{\beta}$, this zero divisor is 
nonspecial.  Hence we can argue just as in the proof of the reduction of 
\eqref{intid}--\eqref{intid'} to \eqref{res=} above that to show 
\eqref{intid2}, it suffices to verify the equality of residues at 
$p_{0} = \lambda_{\alpha}$ for the two sides of \eqref{intid2} for 
$\alpha = 1, \dots, N$.
Again making use of the 
general identity $E_{\Delta}(x,y) = - E_{\Delta}(y,x)$ as well as the 
local development \eqref{E-local} for the prime form, one can  
check that the residues of the left and right hand sides at the simple pole 
$\lambda_{\alpha}$ have the common value
$$
-\frac{\prod_{j} E_{\Delta}(\mu_{j}, 
\lambda_{\alpha})}{ \prod_{i \ne \alpha} E_{\Delta}(\lambda_{i}, 
\lambda_{\alpha})}.
$$
for each $\alpha = 1, \dots, N$.  Putting all the pieces together, we 
see that we have now verified the identity \eqref{intid}, providing an 
independent proof of the formula \eqref{autint} for the solution of 
the First Interpolation Problem for this simplest case.

A similar phenomenon occurs already in the genus $g=0$ case. 
Given specified distinct simple-multiplicity poles $\{\mu_{1}, \dots, \mu_{N}\}$
and distinct simple-multiplicity zeros $\{\lambda_{1}, \dots, 
\lambda_{N}\}$ in the complex plane ${\mathbb C}$, it is completely 
elementary to write down the associated rational zero-pole interpolant
\begin{equation}   \label{sol1}
f(z) = \frac{ \prod_{i=1}^{N} (z - \lambda_{i})}{ \prod_{j=1}^{N} (z - 
\mu_{j})},
\end{equation}
the genus-0 analogue of the prime-form solution of the problem given 
in \cite{MumfordI}.
On the other hand, one can solve the same problem in realization or 
partial fraction form
\begin{equation}  \label{sol2}
  f(z) = 1 + \begin{bmatrix} 1 & \cdots & 1 \end{bmatrix}
  \begin{bmatrix} (z - \mu_{1})^{-1} & & \\ & \ddots & \\ & & (z - 
      \mu_{N})^{-1} \end{bmatrix} S^{-1} \begin{bmatrix} 1 \\ \vdots 
      \\ 1 \end{bmatrix}
\end{equation}
where $S =- \left[ \frac{1}{ \mu_{j} - \lambda_{i} } \right]_{i,j = 
1, \dots, N}$ is the {\em Cauchy matrix} (see \cite[Theorem 
4.3.2]{BGR}), the genus-0 analogue of the formula \eqref{autint}.  Note that in the 
genus-0 case, there is no analogue of the side constraint \eqref{BCVsol}.
The fact that \eqref{sol1} and \eqref{sol2} agree can be seen via 
explicit inversion of the Cauchy matrix (see \cite{Schechter} as well 
as \cite[Lemma A.1.5]{BGR}).
}\end{remark}

\subsection{Appendix: Explicit formulas for building-block functions $f^{D_{0}}_{kw}$}
\label{S:explicit}  
 For completeness we now review here results on theta functions and the prime 
 form; for more complete details we refer to \cite{AlpayVinnikov, 
 Fay, FarkasKra, MumfordII}
 
We assume that we are given a compact Riemann surface $M$.  We let 
$\Delta$ be a line bundle of half-order differentials, i.e., $\Delta 
\otimes \Delta \cong K_{M}$ where $K_{M}$ is the canonical line bundle 
with local holomorphic sections equal to holomorphic differentials  
on $M$.  Assume also that we are given a holomorphic complex 
bundle $E$ of rank $r$ and degree $0$ over $M$ such that $h^{0}(E 
\otimes \Delta) = 0$.  As necessarily ${\rm deg}\, \Delta = g-1$, it 
follows from the discussion in Section \ref{S:4.2} that $E$ has 
property NSF (with test line bundle $\lambda$ taken to be equal to 
$\Delta$) and so in particular $E$ is flat with flat factor of 
automorphy denoted by $\chi$.  We let $\pi_{1} \colon M \times M \to 
M$ be the projection map onto the first component and $\pi_{2} \colon M 
\times M \to M$ be the projection map onto the second component.  The 
defining property for the {\em Cauchy kernel} $K(\chi, \cdot, \cdot)$
is that $K(\chi; \cdot, \cdot)$ be a meromorphic mapping between the vector 
bundle $\pi_{2}^{*}E$ and $\pi_{1}^{*}E \otimes \pi_{1}^{*}\Delta \otimes \pi_{2}^{*} \Delta$ 
on $M \times M$ which is holomorphic outside of the diagonal 
${\mathfrak D} = \{(p,p) \in M \times M \colon p \in M \}$, and that 
the singularity of $K(\chi; \cdot, \cdot)$ on the diagonal  be a 
simple pole with residue equal to the identity $I_{r}$.  It is 
straightforward to show by making use of the assumption that $h^{0}(E 
\otimes \Delta) = 0$ that such a Cauchy kernel is unique if it exists.
The main result of Section 2 of \cite{BV} is that the Cauchy kernel 
indeed does exists (se also \cite{HIP} for an alternative derivation  
in terms of a representation of the Riemann surface as the 
normalizing Riemann surface for an algebraic curve ${\mathbf C}$ embedded in 
projective space ${\mathbb P}^{2}$ with the representation of the bundle 
$E$ as a kernel bundle (up to a twist) associated with a determinantal representation 
for the defining polynomial of the curve ${\mathbf C}$).

Let us next restrict to the case where $E$ is taken to be a line 
bundle. Without loss of generality, we may assume 
that $\chi$ is a flat unitary line bundle (see \cite[page 4]{Fay}).    
Then in this case there is an explicit formula for $K(\chi; 
\cdot, \cdot)$ in terms of theta functions which we now describe.  
We first need  to recall some Riemann-surface function theory.  

We mark $M$ by fixing a canonical $A_{1}, \dots, A_{g}, B_{1}, \dots, B_{g}$
basis for the the homology of $M$.  We  then also fix a normalized basis for holomorphic 
differentials on $M$, 
where normalized means that $\int_{A_{j}} \omega_{i} = \delta_{ij}$.
Then the $B$-period matrix $\Omega$ for $M$ is defined as the matrix 
with $j$-th column equal to $\sbm{ \int_{B_{j}} \omega_{1} \\ \vdots 
\\ \int_{B_{j}} \omega_{g}}$.  The matrix $\Omega$ has positive 
imaginary part and the Jacobian variety of $M$ is 
defined to be the quotient space $J(M) = {\mathbb C}^{g}/ ({\mathbb 
Z}^{g} + \Omega {\mathbb Z}^{g})$.  We fix a base point $p_{0} \in 
M$ and then define the Abel-Jacobi map $\phi \colon M \to J(M)$ by
\begin{equation}   \label{AbelJacobi}
\phi(p) = \left( \int_{p_{0}}^{p} \omega_{1}, \dots, 
  \int_{p_{0}}^{p} \omega_{g} \right).
\end{equation}
We write $\theta(z)$ for the theta function associated with the 
lattice ${\mathbb Z}^{g} + \Omega {\mathbb Z}^{g}$ where $\Omega$ is 
the period matrix of $M$ (the multivariable version of the classical 
theta function already introduced above in \eqref{theta}), namely
$$
\theta(z) = \sum_{m \in {\mathbb Z}^{g}} e^{ \pi i\langle\Omega 
m,m\rangle + 2 \pi i \langle z, m \rangle}.
$$

We will also need the theta function with characteristic, defined for 
$a, b \in ({\mathbb R}/{\mathbb Z})^{g}$ as:
\begin{equation}  \label{thetachar1}
\theta \sbm {a \\ b } (z) =
\sum_{m \in {\mathbb Z}^{g}} e^{\pi i \langle \Omega(m+a), m+a \rangle} 
e^{2 \pi i \langle z+b, m+a \rangle}.
\end{equation}
which can be expressed directly in terms of $\theta$ as
\begin{equation}   \label{thetachar2}
 \theta \sbm{a \\ b }(z) =
 \exp( \pi i \langle \Omega a, a \rangle + 2 \pi i \langle z + b, a 
 \rangle) \theta(z +  \Omega a + b).
 \end{equation}
 In particular, if we set $a=b= 0$, then $\theta\sbm{0 \\ 0}(z) = 
 \theta(z)$.
A direct verification shows that that $\theta \sbm{a \\ b}$ has the 
quasi-periodicity property with respect to the period lattice 
${\mathbb Z}^{g} + \Omega {\mathbb Z}^{g}$ given by
\begin{equation}  \label{theta-quasiper}
    \theta \sbm{a \\ b}(z + m + \Omega n) =
    \exp(2 \pi i \langle  m, a \rangle -2 \pi i\langle  n, b \rangle -
    \pi i \langle \Omega n, n\rangle -  2 \pi i \langle z, n \rangle) 
    \theta \sbm{a \\ b} (z).
\end{equation}

A closely related gadget is a variant of the prime form (defined 
below) given by
$$
    E_{\be}(p,q) = \theta(\be + \phi(p) - \phi(q))
$$
where $\be \in {\mathbb C}^{g}$, $p,q \in M$ and $\theta(\be) = 0$.  By 
following the construction in \cite[pages 158--160]{MumfordI}, one can show 
that, given a divisor 
\begin{equation}  \label{prescribedD}  
 D = \lambda_{1} + \cdots + \lambda_{n} - \mu_{1} - \cdots - \mu_{n} 
\end{equation}
of degree zero, for an appropriate choice of $\be$ with $\theta(\be) 
= 0$, the function $f$ given by 
$$
 f(p) = \frac{ \prod_{i=1}^{n} E_{\be}(p, \lambda_{i})}{\prod_{j=1}^{n} 
 E_{\be}(p, \mu_{i})}
$$
is relatively automorphic on $\widehat M$ with divisor $\widehat D$ 
having projection down to $M$ equivalent to the prescribed divsor $D$ (4.64) 
and has flat factor of automorphy holomorphically equivalent to the 
factor of automorphy $\chi$ whose action on the canonical basis for the homology of $M$ 
is given by
\begin{align*}
& \chi(A_{\ell}) = e^{- 2 \pi i a_{\ell}}, \quad \chi(B_{\ell}) = e^{ 2 
 \pi i b_{\ell}} \text{ for } \ell = 1, \dots, g \text{ where } \phi(D) = \Omega a +  b, \\
 & a =(a_{1}, \dots, a_{g}) \text{ and } b = (b_{1}, \dots, b_{g}) \in 
 {\mathbb R}^{g}/{\mathbb Z}^{g}.
\end{align*}
In this way we may associate a flat factor of automorphy and 
associated line bundle $L_{\ba}$ with a point $\ba = \Omega a + b $
(modulo the period lattice ${\mathbb Z}^{g} + \Omega {\mathbb Z}^{g}$)
in the Jacobian $J(M)$ of $M$.

An important special case of theta function with characteristic 
\eqref{thetachar1} is the situation 
where the characteristic components $a,b$ are taken to be in $(\frac{1}{2} 
{\mathbb Z}/{\mathbb Z})^{g}$
(half-order characteristic):  for the case where say $a_{*},b_{*} \in
(\frac{1}{2} {\mathbb Z}/{\mathbb Z})^{g}$, the associated theta 
function with half-order characteristic $\theta\sbm{a_{*}\\ b_{*}}$ 
has the property that its zero set is invariant under the symmetry $z 
\mapsto -z$:
$$ \theta\sbm{a_{*}\\ b_{*}}(z) = 0 \Leftrightarrow 
\theta \sbm{a_{*}\\ b_{*}}(-z) = 0.
$$
This happens in exactly two possible ways:  either $\theta \sbm{a_{*} 
\\ b_{*}}$ is even 
($\theta \sbm{a_{*}\\ b_{*}}(-z)$
$= \theta \sbm{a_{*}\\ b_{*}}(z)$) 
or $\theta \sbm{a_{*}\\ b_{*}}$ is odd
($\theta \sbm{a_{*}\\ b_{*}}(-z)$ $= - \theta \sbm{a_{*}\\ 
b_{*}}(z)$).  In the first case $(a_{*},b_{*})$ is said to be an even 
half-order characteristic and in the second case $(a_{*},b_{*})$ is 
said to be an odd half-order characteristic.  An important result is 
that nonsingular odd half-order characteristics always exist, i.e., 
an odd half-order characteristic $(a_{*},b_{*})$ for which in addition
the differential $d \theta\sbm{a_{*} \\ b_{*}}(0)$  is not zero
(see Lemma 1 page 3.208 of \cite{MumfordII}).

Let us now fix a choice $(a_{*},b_{*})$ of nonsingular odd half-order 
characteristic. Since the half-order characteristic $(a_{*},b_{*})$ 
is nonsingular, Riemann's zero theorem (see e.g.\ \cite[page 
149]{MumfordI} or \cite[pages 308--309]{FarkasKra}) implies that 
$\theta\sbm{a_{*}\\ b_{*}}(\phi(q) - \phi(\cdot))$ has precisely $g$ 
zeros $p_{1}, \dots, p_{g}$ which are uniquely determined by  the 
equality
$$
\phi(p_{1}) + \cdots + \phi(p_{g}) =  - \kappa_{0} + \phi(q) + \Omega a_{*} 
+ b_{*}
$$
where $\kappa_{0}$ is Riemann's constant.
Since $(a_{*},b_{*})$ is odd, one of these zeros is $q$. Hence
without loss of generality we may take $p_{g} = q$ and we are left with
$$
\phi(p_{1}) + \cdots + \phi(p_{g-1}) =  - \kappa_{0} + \Omega a_{*}  + 
b_{*}.
$$
As $ 2 (\Omega a_{*} + b_{*}) = 0$ in $J(M)$, 
multiplying this last expression by 2 converts it to
\begin{equation}   \label{2D=0}
2 (\phi(p_{1}) + \cdots +  \phi(p_{g-1}) + \ba_{*}) = - 2 \kappa_{0}
\end{equation}
where we set $\ba_{*} = \Omega a_{*} + b_{*}$.
It is known that a divisor $K$ is the divisor of a meromorphic 
differential on $M$ if and only if $\phi(K) = -2 \kappa$
(see \cite[page 318]{FarkasKra}). 
We let $\Delta_{*}$ be the line bundle associated with the divisor 
$p_{1} + \cdots + p_{g-1}$.  We now conclude from the equality 
\eqref{2D=0} that  
$$
 ( L_{\ba_{*}} \otimes \Delta_{*} ) \otimes (L_{\ba_{*}} \otimes 
 \Delta_{*}) = K_{M}
$$
where $K_{M}$ is the line bundle associated with the canonical divisor 
on $M$. Thus the bundle $\Delta : = L_{\ba_{*}} \otimes \Delta_{*}$ is 
a bundle of half-order differentials; we shall henceforth assume that 
{\em the bundle of half-order differentials in the definition of the 
Cauchy kernel arises in this way from a half-order characteristic 
$(a_{*}, b_{*})$.}
The fact that $d\theta \sbm{a_{*} \\ b_{*}}(0)$ is nonzero 
also implies that the bundle $\Delta$ has only one nonzero 
holomorphic section (up to a multiplicative constant) (see 
\cite[page  293]{AlpayVinnikov}).  We let 
$\sqrt{\xi_{*}(p)}$ denote the choice such that
$$
  \left( \sqrt{\xi_{*}(p)} \right)^{2} = \sum_{j=1}^{g} 
  \frac{\partial \theta \sbm{a_{*} \\ b_{*}}}{\partial z_{j}}(0) 
  \omega_{j}(p).
$$
We now define the
{\em prime form} $E_{\Delta}(p,q)$ by
$$
     E_{\Delta}(p,q) = \frac{\theta \sbm{a_{*} \\ b_{*}}( \phi(q) - \phi(p))}
     {\sqrt{\xi_{*}(p)} \sqrt{\xi_{*}(q)}}.
$$
If we use $\xi_{*}(p) = \sum_{j=1}^{g} \frac{\partial 
\theta\sbm{a_{*}\\ b_{*}}}{\partial z_{j}}(0) \omega_{j}(p)$ as a local coordinate $t = 
t(p)$, then $E_{\Delta}$ has a local representation at 
points $(p,q)$ near $p = q$ given by
\begin{equation}   \label{E-local}
    E_{\Delta}(p,q) = \frac{t(q) - t(p)}{\sqrt{dt(p)} \sqrt{dt(q)}}  
    (1 + O(t(p) - t(q))^{2} )
\end{equation}
(see formula (26) page 19 of \cite{Fay}, a corrected version of 
Property 4 page 3.210 of \cite{MumfordII}).

Let now $\chi$ be a general rank-1 unitary factor of automorphy, with
action  on the 
generators $A_{1}, \dots, A_{g}, B_{1}, \dots, B_{g}$ for the 
homology basis for $M$ given by
\begin{equation} \label{A-a-ell}
\chi(A_{\ell}) = \exp (-2 \pi i a_{\ell}), \quad
\chi(B_{\ell}) = \exp(2 \pi i b_{\ell})
\end{equation}
for $\ell = 1, \dots, g$ where   $a,b \in ({\mathbb 
R}/{\mathbb Z})^{g}$.
It then can be checked that the Cauchy kernel $K(\chi; \cdot, \cdot)$  
is given explicitly by 
\begin{equation}  \label{Cauchykernelform}
    K(\chi; p,q) = \frac{ \theta \sbm{ a \\ b }
    (\phi(q) - \phi(p))}{ \theta \sbm{a \\ b}(0) E_{\Delta}(q,p)}
\end{equation}
(see \cite{HIP, BV}).

By using the formula \eqref{thetachar2} 
for $\theta \sbm{a \\ b}$,  we can express $K(\chi; p,q)$ in terms of 
the theta function $\theta$ itself:
\begin{equation}  \label{Cauchykernelform2}
    K(\chi; p,q) =  \frac{\exp(2 \pi i a^{\top} (\phi(q) - \phi(p)) 
  \cdot  \theta(\phi(q) - \phi(p) + \be)}{\theta(\be) \cdot E_{\Delta}(q,p)}
\end{equation} 
where we use the notation
\begin{equation} \label{be} 
 \be =  b + \Omega a.
\end{equation}

As the following result shows, the Cauchy kernel and the building-block functions $f^{D_{0}}_{k w}$  
used to build the matrices of the form $f^{D_{0}}_{w, A}$ \eqref{fD0wA} are closely related.

\begin{thm}  \label{T:fw-form}  Suppose that $D_{0} = p_{1} + \cdots +
    p_{g} - p_{0}$ is the divisor on $M$ where the divisor $D= p_{1} + 
    \cdots + p_{g}$ is nonspecial.  Write the associated line bundle $\lambda_{D_{0}}$ 
    in the form
    \begin{equation}   \label{bundleid}
    \lambda_{D_{0}} = \zeta_{0} \otimes \Delta
    \end{equation}
    for a degree-0 necessarily flat line bundle $\zeta_{0}$, where 
    $\Delta$ is the bundle of half-order differentials used in the 
    prime form $E_{\Delta}(\cdot, \cdot)$.  Then 
      the canonical function 
    $f_{w}^{D_{0}}(p)$ associated with $D_{0}$ as introduced in Section \ref{S:4.2} 
    is given  by
    \begin{equation}   \label{fD0wfor}
	f^{D_{0}}_{w}(p) = \frac{K(\zeta_{0}; p,w)}{K(\zeta_{0}; p, p_{0})} 
	K(\zeta_{0}; w, p_{0})
    \end{equation}
 or in terms of theta functions,
  \begin{equation}   \label{fD0wfor'}
	f^{D_{0}}_{w}(p) = \frac{1}{\theta(\be)} \frac{\theta(\phi(w) 
	- \phi(p) + \be)}{E_{\Delta}(w,p)} 
	\frac{ E_{\Delta}(p_{0},p)}{\theta(\phi(p_{0}) - \phi(p) + 
	\be)}  \frac{\theta(\phi(p_{0}) - \phi(w) + 
	\be)}{E_{\Delta}(p_{0},w)}.
\end{equation}
 Furthermore, for $k > 1$, the function $f^{D_{0}}_{k w}(p)$ is 
    given by
    \begin{equation} \label{fD0kwfor}
	f^{D_{0}}_{kw}(p) = \frac{1}{(k-1)!}\frac{d^{(k-1)}}{dw^{(k-1)}} 
	f^{D_{0}}_{w}(p).
    \end{equation}
   \end{thm}
   
   \begin{proof}  Let us set the right hand side of \eqref{fD0wfor} 
       equal to $\widetilde f_{w}^{D_{0}}(p)$. It is easily seen that the right-hand side of 
       \eqref{fD0wfor} is automorphic as a function of $p$ and has a 
       pole at $w$ with residue 1. Use $p_{0}$ as the base point for 
       the Abel-Jacobi map $\phi$.  Note that the zero divisor 
       $D^{p_{0}}$ is equal to the zero divisor of $\theta\sbm{a \\ 
       b}\circ( \phi(\cdot))$ which is also arranged to be the same 
       as the divisor $D$.  Thus $D^{p_{0}} = D$.
      By definition we also know that $K(\zeta_{0}; 
       \cdot, p_{0})$ has a pole at $p_{0}$.  From the formula 
       \eqref{fD0wfor} we conclude that $(\widetilde f_{w}^{D_{0}}) + w + D \ge 
       0$.  Note also that $\widetilde f_{w}^{D_{0}}$ is normalized 
       so that the pole at $w$ has residue equal to 1.  By the 
       uniqueness of the solution of the properties defining 
       $f_{w}^{D_{0}}$ we conclude that $f_{w}^{D_{0}} = \widetilde 
       f_{w}^{D_{0}}$.  Substitution of the  formula 
       \eqref{Cauchykernelform2} for the Cauchy kernel (three times) 
       converts the formula \eqref{fD0wfor} to \eqref{fD0wfor'}.
           
 The formula for the higher multiplicity case is checked 
       similarly.
    \end{proof}
    
    The following corollaries  are of interest for us in the body of the paper.
    
\begin{cor}  \label{C:pvsp0}
    When the building-block function $f^{D_{0}}_{w}(p) = f^{D - 
    p_{0}}_{w}(p)$ is considered as a function of $p_{0}$ for each 
    fixed $w$ and $p$, then $f^{D_{0}}_{w}(p) = f^{D -p_{0}}_{w}(p)$
    is a single-valued meromorphic function on $M$ with only pole a 
    simple pole at $w$ having residue there given by
 \begin{equation}   \label{residue}
     {\rm res}_{p_{0} = w} f^{D -p_{0}}_{w}(p) = -1.
   \end{equation}
 \end{cor}
   
   \begin{proof}  Use either of formulas \eqref{fD0wfor}, 
       \eqref{fD0wfor'}
       to see that $f_{w}^{D - p_{0}}(p)$ as a function of 
       $p$ has a pole of  residue 1 at $p = w$ while as a 
       function of $p_{0}$ there is a pole at $p_{0} = w$ of residue 
       $-1$.  In fact one can make the identification
       $$ 
       f_{w}^{D - p_{0}}(p) = - f_{w}^{D' - p}(p_{0}).
       $$
      where $D$ is the zero divisor of $K(\zeta_{0}; \cdot, p_{0})$
      (or equivalently of $\theta\sbm{a \\ b}(\phi(p_{0}) - \phi(\cdot) 
      + \be)$)
       while $D'$ is the zero divisor of $K(\zeta_{0}; p, \cdot)$
       (or equivalently of $\theta \sbm{a \\ b}(\phi(\cdot) - \phi(p) 
       + \be)$).
  \end{proof}

\begin{cor}  \label{C:contdep}
    The matrix $\Gamma^{D_{0}}_{\cD}$ depends continuously on the 
    support of the divisor $D$.
\end{cor}

\begin{proof} We fix the base point $p_{0}$ and write $D_{0} = D - 
    p_{0}$. From the identification \eqref{bundleid}, we have
    $$
      D_{0} = ( \zeta_{0}) + (\Delta)
 $$
 where $(\Delta)$ (the divisor of the bundle $\Delta$ of half-order 
 differentials) is fixed.  The divisor $(\zeta_{0})$ of the flat line 
 bundle $\zeta_{0}$ in turn is specified by
 $$
   \phi((\zeta_{0})) = \Omega a + b
 $$
 where $a,b \in {\mathbb R}^{g}$ are such that the Cauchy kernel 
 $K(\zeta_{0}; p,q)$ is given by the right-hand side of 
 \eqref{Cauchykernelform}.  Without loss of generality we can fix all 
 the poles in the divisor $(\zeta_{0})$ to be also at the base point 
 $p_{0}$.  From formula \eqref{thetachar1} and the 
 fact that the theta function is infinitely differentiable in its 
 arguments, we se that $\theta\sbm{a \\ b}(\phi(q) - \phi(p))$ is 
 continuous with respect to the parameters $a,b$.  From formulas 
 \eqref{Cauchykernelform} and \eqref{fD0wfor} (and more generally 
 \eqref{fD0kwfor}), we see that $f_{kw}^{D_{0}}$ is continuous with 
 respect to the parameters $a,b$.  Then by Jacobi inversion, we see 
 that $f_{kw}^{D_{0}}$ depends continuously on the support of the 
 divisor $((\zeta_{0}))$, and hence also on the support of the 
 divisor $D_{0}$.  As we are fixing the base point, we then have 
 continuous dependence on the support of the divisor $D$.
\end{proof}

 \end{document}